\documentclass{article}

\usepackage{epsfig}
\usepackage{graphicx}
\usepackage{amsbsy}
\usepackage{amsmath}
\usepackage{amsfonts}
\usepackage{amssymb}
\usepackage{textcomp}
\usepackage{hyperref}
\usepackage{aliascnt}

\newcommand{\mcm}[3]{\newcommand{#1}[#2]{{\ensuremath{#3}}}} 
\newcommand{\sm}{\setminus}
\newcommand{\se}{\subseteq}
\newcommand{\ct}{^\complement}

\mcm{\tuple}{1}{\langle #1 \rangle}
\mcm{\name}{1}{\ulcorner #1 \urcorner}
\mcm{\Nbb}{0}{\mathbb{N}}
\mcm{\Zbb}{0}{\mathbb{Z}}
\mcm{\Rbb}{0}{\mathbb{R}}
\mcm{\Cbb}{0}{\mathbb{C}}
\mcm{\Fbb}{0}{\mathbb{F}}
\mcm{\Bcal}{0}{\cal B}
\mcm{\Ccal}{0}{\cal C}
\mcm{\Dcal}{0}{\cal D}
\mcm{\Ecal}{0}{\cal E}
\mcm{\Fcal}{0}{\cal F}
\mcm{\Gcal}{0}{\cal G}
\mcm{\Hcal}{0}{\cal H}
\mcm{\Ical}{0}{\cal I}
\mcm{\Lcal}{0}{\cal L}
\mcm{\Mcal}{0}{\cal M}
\mcm{\Ncal}{0}{\cal N}
\mcm{\Pcal}{0}{{\cal P}}
\mcm{\Scal}{0}{{\cal S}}
\mcm{\Tcal}{0}{{\cal T}}
\mcm{\Ucal}{0}{{\cal U}}
\mcm{\Vcal}{0}{{\cal V}}
\mcm{\Wcal}{0}{{\cal W}}
\mcm{\Xcal}{0}{{\cal X}}
\mcm{\Ycal}{0}{{\cal Y}}
\mcm{\Zcal}{0}{{\cal Z}}
\mcm{\Mfrak}{0}{\mathfrak M}

\mcm{\restric}{0}{\upharpoonright}
\mcm{\upset}{0}{\uparrow}
\mcm{\downset}{0}{\downarrow}
\mcm{\onto}{0}{\twoheadrightarrow}
\mcm{\smallNbb}{0}{{\small \mathbb{N}}}
\DeclareMathOperator{\preop}{op}
\mcm{\op}{0}{^{\preop}}
\newcommand{\supth}{\textsuperscript{th}}

\newcommand{\itum}{\item[$\bullet$]}

{\begin{array}{c}
\setlength{\unitlength}{1em}}%
{\end{array}}

\usepackage{amsthm}

\newcommand{\theoremize}[2]{\newaliascnt{#1}{thm} \newtheorem{#1}[#1]{#2} \aliascntresetthe{#1}}

\theoremstyle{plain}
\newtheorem{thm}{Theorem}[section]
\theoremize{lem}{Lemma}
\theoremize{sublem}{Sublemma}
\theoremize{claim}{Claim}
\theoremize{rem}{Remark}
\theoremize{prop}{Proposition}
\theoremize{cor}{Corollary}
\theoremize{que}{Question}
\theoremize{oque}{Open Question}
\theoremize{con}{Conjecture}

\theoremstyle{definition}
\theoremize{dfn}{Definition}
\theoremize{eg}{Example}
\theoremize{exercise}{Exercise}
\theoremstyle{plain}

\usepackage{verbatim}

\title{\scshape Infinite Matroids and Determinacy of Games}
\author{Nathan Bowler \and Johannes Carmesin}

\begin{document}

\maketitle

\begin{abstract}
Solving a problem of Diestel and Pott, we construct a large class of infinite matroids.
These can be used to provide counterexamples against the natural extension of the Well-quasi-ordering-Conjecture to infinite matroids and to show that the class of planar infinite matroids does not have a universal matroid.

The existence of these matroids has a connection to Set Theory in that it corresponds to the Determinacy of certain games.
To show that our construction gives matroids, we introduce a new very simple axiomatization of the class of countable tame matroids. 


\end{abstract}

\section{Introduction}

One of the big problems in the development of infinite matroid theory has been 
that there have not been very many known examples of infinite matroids, which then could have been used as a supply of counterexamples. 
Recent work of Diestel and Pott \cite{DP:dualtrees} suggested somewhere to look for a new large class of infinite matroids.
Before we will talk about the class itself, we shall first explain a little bit about what they were doing.

They looked at the question how one could extend the following theorem to infinite graphs:
Two finite graphs are dual if and only if their cycle matroids are dual to each other.
In the infinite case, the situation is no longer that easy since there are at least two different cycle matroids associated to an infinite locally finite graph $G$: the finite cycle matroid $M_{FC}(G)$,
whose circuits are the finite circuits of $G$, and the topological cycle matroid $M_C(G)$, whose circuits are edge sets of topological circles in the end-compactification $|G|$
of $G$ \cite{matroid_axioms}.
Note that $M_{FC}(G)$ is finitary and $M_C(G)$ is cofinitary.
In fact, if $G$ and $G^*$ are dual in a suitable sense, then $M_{FC}(G)$ and $M_C(G^*)$
are dual to each other.

Motivated by the slight asymmetry of this fact, Diestel and Pott \cite{DP:dualtrees} introduced 
a more general context in which a stronger result is true.
Given a partition of the ends of $G$ into $\Psi$ and $\Psi\ct$, a $\Psi$-circuit is a topological circuit using only ends from $\Psi$, and a $\Psi$-tree is a set of edges maximal with the property that it does not include a $\Psi$-circuit.
If $\Psi=\Omega(G)$, then the $\Psi$-circuits and $\Psi$-trees are the $M_C(G)$-circuits and $M_C(G)$-bases, whereas if $\Psi=\varnothing$, then the $\Psi$-circuits and $\Psi$-trees are the $M_{FC}(G)$-circuits and $M_{FC}(G)$-bases.

Let $G=(V,E,\Omega)$ and $G^*=(V^*,E,\Omega)$ be two finitely separable\footnote{
A graph is \emph{finitely separable} if any two vertices can be separated by removing only finitely many vertices} $2$-connected graphs with the same set of edges $E$ and the same set of ends $\Omega$.
Diestel and Pott showed that if $G$ and $G^*$ are duals, then for every $\Psi$ 
the complements of $\Psi$-trees in $G$ are precisely the $\Psi\ct$-trees in $G^*$.
This means that if the set of $\Psi$-trees were the set of bases of some matroid, then
the set of $\Psi\ct$-trees in $G^*$ would also be the set of bases of a matroid, namely its dual.
This tempted Diestel and Pott to ask\footnote{personal communication} the following.
\begin{que}\label{que:is_matroid}
Let $G$ be a locally finite graph and $\Psi\se \Omega(G)$.
Is the set of $\Psi$-trees the set of bases of a matroid?
\end{que}

Unfortunately, the answer to this question is no. Indeed, with some effort the question
can be reduced to the question 
about path-connectedness in certain connected subspaces of $|G|\sm \Psi\ct$.
Questions of this type have been considered by Georgakopoulos in \cite{cnpc}, and his main counterexample from there also gives a counterexample here.
However, the construction of the set $\Psi$ in this case heavily relies on the Axiom of Choice (we will return to this point later). 

The purpose of this paper is to show that if the set $\Psi$ is pleasant enough, in a sense we will now explain, then the set of $\Psi$-trees is the set of bases of a matroid.

It will turn out that the way pleasantness is measured has to with Determinacy of Sets (See Section \ref{sec:2sums} for an explanation why this is a good way to measure pleasantness here).
Determinacy of sets is usually defined using games.
Let $\Psi\se A^\Nbb$ for some set $A$, then the  $\Psi$-game $\Gcal(\Psi)$ is the following game between two players
which has one move for every natural number.
In each odd move the first player chooses an element of $A$ whereas in each even move the second player chooses such an element. The first player wins if and only if the sequence they generate between them is in $\Psi$.
The set $\Psi$ is \emph{determined} if one player has a winning strategy.
The question which sets are determined has been investigated a lot in set theory \cite{determinacy_survey}:
The statement that all subsets $\Psi\se A^\Nbb$ with $A$ countable are determined is called
\emph{the Axiom of Determinacy}, and is sometimes taken as an alternative  to the Axiom of Choice. Indeed, if one assumes the Axiom of Determinacy instead of the Axiom of Choice, every set of real numbers becomes Lebesgue measurable \cite{determinacy_to_measurability}.
A deep result in this area says that if $\Psi$ is Borel (in the product topology), then it is determined \cite{Martin_borel}. 

We will want to consider slightly more general games in which 
the set of moves available to a player may vary depending on the moves made so far in the game, and may even sometimes be empty.
Any game like this can be coded up by an equivalent game of the above type, so we will not worry too much about this issue.
A game is \emph{determined}
if at least one of the players has a winning strategy.

Next we sketch how we transform \autoref{que:is_matroid} into an equivalent statement about determinacy of games.
First we build from a given locally finite graph $G$ what we call a \emph{tree of matroids}
which is a tree $T$ whose ends are the ends of $G$, where for each node we store a finite matroid, and for each edge we store information about how to glue together the matroids for the two incident nodes. We do this in such a way that if we do all the gluing at once we get  back all the relevant information about $G$.

Then we introduce the \emph{circuit games} which are games of the above type
in which each possible play defines a (possibly infinite) path in $T$ starting at a fixed node of $T$. If play continues forever, then the path is infinite and the first player wins if and only if that path belongs to some end in $\Psi$ (for a precise Definition of the game see Section \ref{sec:2sums} or \ref{sec:arbi_sums}).
Having done this, we then are able to reduce \autoref{que:is_matroid}
to a question about the determinacy of circuit games:
\begin{thm}\label{main}
 The set of $\Psi$-trees is the set of bases of a matroid
if and only if certain circuit games are all determined.
\end{thm}

Applying the determinacy of Borel sets mentioned above, we obtain the following.

\begin{cor}
Let $G$ be a locally finite graph and $\Psi\se \Omega(G)$ a Borel set.
 Then the set of $\Psi$-trees is the set of bases of a matroid.
\end{cor}

A key ingredient in proving \autoref{main} is the following theorem which comes from a new axiomatisation of the class of countable tame matroids. Here a matroid is tame  if every circuit-cocircuit-intersection is finite.

\begin{thm}\label{new_axioms}
 Let $E$ be countable and $\Ccal,\Dcal\se \Pcal(E)$ such that $|C\cap D|$ is never $1$ or infinite for any $C\in \Ccal$ and $D\in \Dcal$.
Further assume that for all partitions $E=P_{Co}\dot\cup P_{De}\dot\cup \{e\}$ where $e\in E$
either $P_{Co}+e$ includes an element of $\Ccal$ through $e$ or
$P_{De}+e$ includes an element of $\Dcal$ through $e$.

Then the minimal nonempty element of $\Ccal$ and $\Dcal$ are the circuits and cocircuits of some matroid. 
\end{thm}

These new axioms are simpler than the general matroid axioms since there is no axiom that is as complicated as the axiom called $(IM)$. It should be noted that this axiomatisation is very similar to Minty's axiomatisation of finite matroids
\cite{Minty}.

So far we have talked only about locally finite graphs. 
Our proof in that case heavily relies on the assumption that the graph is locally finite (Once the definition of a tree of matroids is made precise, it is clear that this requires the graph to be locally finite). However, we are able to extend our results to all countable graphs. The argument takes the whole of Section \ref{loc_to_count} and uses a new technique; we expect that this technique can also be used in other contexts to extend results from locally finite to countable graphs.

The new matroids we constructed in this paper can be used to find counterexamples
to various conjectures about infinite matroids. We shall illustrate this with three examples.

A class $\Fcal$ of matroids is \emph{well-quasi-ordered} 
if for every sequence $(M_n|n\in\Nbb)$ with $M_n\in\Fcal$
there are $i<j$ such that $M_i \preceq M_j$.
Robertson and Seymour proved \cite{MR2099147} that the class of finite graphs is
well-quasi-ordered. 
In 1965, Nash-Williams \cite{MR0175814} proved that infinite trees are well-quasi-ordered.
This was extended by Thomas to the class of graphs of bounded branch width \cite{MR932450}.
On the other hand, he provided a sequence of uncountable graphs showing that the class of all graphs is not well-quasi-ordered \cite{MR913450}. 
It is not known if the class of countable graphs is well-quasi-ordered.
For finite matroids, it is at the moment an important project to prove that the class of matroids representable over a fixed finite field is well-quasi-ordered.
Geelen, Gerards and Whittle \cite{MR1889259} proved that this is true if the matroids have bounded branch width.
For infinite matroids almost nothing is known. 
Azzato and  Jeffrey \cite{MR2771295} 
made a first step towards proving that the class of finitary matroids of bounded branch width representable over a fixed finite field is well-quasi-ordered.
In this paper we consider the corresponding question for infinite matroids, not just for the finitary ones.
The new matroids we construct can be used to show that the answer to this question is no, even in a very special case.

\begin{cor}\label{wqo}
 The countable binary matroids of branch-width at most $2$ are not well-quasi-ordered (under the minor relation).
\end{cor}

The next conjecture concerns the number of possible non-isomorphic matroids on a countable ground set. Clearly, there cannot be more than $2^{2^{\aleph_0}}$. We show that this bound is actually attained.

\begin{cor}\label{non_iso_cor}
There are $2^{2^{\aleph_0}}$ non-isomorphic tame matroids with no $M(K_4)$-minor and no $U_{2,4}$-minor on a countable ground set. 
\end{cor}

Diestel and K\"uhn \cite{MR1709660} proved that there is a countable planar graph that has all other
countable planar graphs as minors. Such a graph is called a \emph{universal countable planar graph (with respect to the minor relation)}.
In the same spirit, we call a matroid \emph{universal for a class $\Fcal$ of matroids (with respect to the minor relation)} if it is in $\Fcal$ and it has every member of $\Fcal$ as a minor.
A matroid is \emph{planar} if it is tame and all its finite minors are planar \cite{BCC:graphic_matroids}.
The result of    Diestel and K\"uhn does not extend to infinite matroids: 

\begin{cor}\label{uni_cor}
 There is no universal matroid for the class of countable planar matroids. 
\end{cor}

This paper is organised as follows. 
In \autoref{prelims}, we sum up the basic definitions and facts. 
In \autoref{sec:psi}, we explain how the matroids arise from infinite graphs and their ends even if those graphs are not finitely separable.
The new axioms for countable tame matroids are introduced in \autoref{orth}.
At the end of that section we explain how Georgakopoulos' construction 
can be used to get a counterexample against \autoref{que:is_matroid}.

The proof of \autoref{main} takes the next 4 sections.
The proof can be subdivided into two parts: first we construct from
a given graph a tree of matroids and analyse it.
Then we use this tree of matroids as a tool to prove \autoref{main}.

Here, the purpose of \autoref{treesofmatroids} and \autoref{sec:2sums} is to prove \autoref{main}
in a special case. Although this special case is much simpler than the general case, many ideas are already visible
there.
In \autoref{tm2} and \autoref{sec:arbi_sums}, we prove \autoref{main} in the general case.
While \autoref{treesofmatroids} is concerned with trees of matroids arising in the special case,
in \autoref{tm2} we show how one has to extend the method for the general case.
The other two sections are both considered with the second part of the proof and bear a similar relation.
In \cite{BC:inter_psi}, we shall give an alternative new proof of \autoref{main}.

In \autoref{loc_to_count}, we deduce the countable case from the locally finite case.
In \autoref{apps}, we prove \autoref{wqo}, \autoref{non_iso_cor} and \autoref{uni_cor}.

\section{Preliminaries}\label{prelims}
Throughout, notation and terminology for (infinite) graphs are those of~\cite{DiestelBook10}, and for matroids those of~\cite{Oxley,matroid_axioms}. In this paper, we only work with simple graphs. However, all the results and proofs can easily be extended to multigraphs. We will rely on the following lemma from \cite{DiestelBook10}:

\begin{lem}[K\"onig's Infinity Lemma \cite{DiestelBook10}]\label{Infinity_Lemma}
Let $V_0,V_1,\ldots$ be an infinite sequence of disjoint non-empty finite sets, and let $G$ be a graph on their union. Assume that every vertex $v$ in $V_n$ with $n\geq 1$ has a neighbour $f(v)$ in $V_{n-1}$. Then $G$ includes a ray $v_0v_1\ldots$ with $v_n\in V_n$ for all $n$.
\end{lem}

For any graphs $G$ and $H$, we will use $G \times H$ to denote the graph with vertex set $V(G) \times V(H)$ and with edge set $$\{e \times \{v\} | e \in E(G), v \in V(H)\} \cup \{\{v\} \times e| v \in V(G), e \in E(H)\}.$$
The edges in $\{e \times \{v\} | e \in E(G), v \in V(H)\}$ are called $G$-edges, and those in $\{\{v\} \times e| v \in V(G), e \in E(H)\}$ are called $H$-edges.

We will also make use of the following elementary fact of Linear Algebra:

\begin{lem}\label{span}
Let $X$ be a finite set of vectors in a finite dimensional vector space $V$,
and let $y\in V$.
 $X^\perp \se \{y\}^\perp$ if and only if $y$ is in the span $\langle X \rangle$ of $X$.
\end{lem}

$M$ always denotes a matroid and $E(M)$ (or just $E$), $\Ical(M)$ and $\Ccal(M)$ denote its ground 
set and its sets of independent sets and circuits, respectively. For the remainder of this section we shall recall some basic facts about infinite matroids.

A set system $\Ical\subseteq \Pcal(E)$ is the set of independent sets of a matroid if and only if it satisfies the following {\em independence  axioms\/} \cite{matroid_axioms}.
\begin{itemize}
	\item[(I1)] $\varnothing\in \Ical(M)$.
	\item[(I2)] $\Ical(M)$ is closed under taking subsets.
	\item[(I3)] Whenever $I,I'\in \Ical(M)$ with $I'$ maximal and $I$ not maximal, there exists an $x\in I'\setminus I$ such that $I+x\in \Ical(M)$.
	\item[(IM)] Whenever $I\subseteq X\subseteq E$ and $I\in\Ical(M)$, the set $\{I'\in\Ical(M)\mid I\subseteq I'\subseteq X\}$ has a maximal element.
\end{itemize}

A set system $\Ccal\subseteq \Pcal(E)$ is the set of circuits of a matroid if and only if it satisfies the following {\em circuit  axioms\/} \cite{matroid_axioms}.
\begin{itemize}
\item[(C1)] $\varnothing\notin\Ccal$.
\item[(C2)] No element of $\Ccal$ is a subset of another.
\item[ (C3)](Circuit elimination) Whenever $X\subseteq o\in \Ccal(M)$ and $\{o_x\mid x\in X\} \subseteq \Ccal(M)$ satisfies $x\in o_y\Leftrightarrow x=y$ for all $x,y\in X$, 
then for every $z \in o\setminus \left( \bigcup_{x \in X} o_x\right)$ there exists a  $o'\in \Ccal(M)$ such that $z\in o'\subseteq \left(o\cup  \bigcup_{x \in X} o_x\right) \setminus X$.

\item[(CM)] $\Ical$ satisfies (IM), where $\Ical$ is the set of those subsets of $E$ not including an element of $\Ccal$.
\end{itemize}


\begin{lem}\label{fdt}
 Let $M$ be a matroid and $s$ be a base.
Let $o_e$ and $b_f$ a fundamental circuit and a fundamental cocircuit with respect to $s$, then
\begin{enumerate}
 \item $o_e\cap b_f$ is empty or $o_e\cap b_f=\{e,f\}$ and
\item $f\in o_e$ if and only if $e\in b_f$.
\end{enumerate}
\end{lem}

\begin{proof}
To see the first note that $o_e\subseteq s+e$ and $b_f\subseteq (E\setminus s)+f$.
So $o_e\cap b_f\subseteq \{e,f\}$. As a circuit and a cocircuit can never meet in only one edge, the assertion follows.

To see the second, first let $f\in o_e$.
Then $f \in o_e \cap b_f$, so by (1) $o_e \cap b_f = \{e, f\}$ and so $e \in b_f$.
The converse implication is the dual statement of the above implication.
\end{proof}

\begin{lem}\label{o_cap_b}
 For any circuit $o$ containing two edges $e$ and $f$, there is a cocircuit $b$ such that $o\cap b=\{e,f\}$.
\end{lem}

\begin{proof}
As $o-e$ is independent, there is a base including $o-e$.
By \autoref{fdt}, the fundamental cocircuit of $f$ of this base intersects $o$ in $e$ and $f$, as desired. \end{proof}

\begin{lem} \label{rest_cir}
 Let $M$ be a matroid with ground set $E = C \dot \cup X \dot \cup D$ and let $o'$ be a circuit of $M' = M / C \backslash D$.
Then there is an $M$-circuit $o$ with $o' \subseteq o \subseteq o' \cup C$.
\end{lem}
\begin{proof}
Let $s$ be any $M$-base of $C$. Then $s \cup o'$ is $M$-dependent since $o'$ is $M'$-dependent.
On the other hand,  $s \cup o'-e$ is $M$-independent whenever $e\in o'$ since $o'-e$ is $M'$-independent.
Putting this together yields that $s \cup o'$ contains an $M$-circuit $o$, and this circuit must not avoid any $e\in o'$, as desired.
\end{proof}

A \emph{scrawl} is a union of circuits. In \cite{BC:rep_matroids},
(infinite) matroids are axiomatised in terms of scrawls. 
The set $\Scal(M)$ denotes the set of scrawls of the matroid $M$.
Dually a \emph{coscrawl} is a union of cocircuits.
Since no circuit and cocircuit can meet in only one element,
no scrawl and coscrawl can meet in only one element. 
In fact, this property gives us a simple characterisation of scrawls in terms of coscrawls and vice versa.

\begin{lem}\label{is_scrawl}\cite{BC:rep_matroids}
Let $M$ be a matroid, and let $w\subseteq E$. The following are equivalent:
\begin{enumerate}
 \item $w$ is a scrawl of $M$.
 \item $w$ never meets a cocircuit of $M$ just once.
 \item $w$ never meets a coscrawl of $M$ just once.
\end{enumerate}
\end{lem}

\begin{proof}
It is clear that (1) implies (3) and (3) implies (2), so it suffices to show that (2) implies (1). 
Suppose that (2) holds and let $e\in w$. Then in the minor $M/(w - e)\setminus(E \setminus w)$ on the groundset $\{e\}$, $e$ cannot be a co-loop, by the dual of \autoref{rest_cir} and (2). So $e$ must be a loop, and by \autoref{rest_cir} there is a circuit $o_e$ with $e \in o_e \subseteq w$. Thus $w$ is the union of the $o_e$, and so is a scrawl.

\end{proof} 

\begin{lem}\label{is_circuit}
 Let $w$ be a dependent set. Then $w$ is a circuit if and only if for any edges $e$ and $f$ of $w$ there is a cocircuit $b$ with $w \cap b = \{e, f\}$.
\end{lem}
\begin{proof}
 The `only if' direction is immediate from \autoref{o_cap_b}. For the `if' direction, pick a circuit $o \se w$. If $o \not = w$ then we can find $e \in o$ and $f \in w \sm o$, and choosing $b$ a cocircuit with $b \cap w = \{e, f\}$, we get $b \cap o = \{e\}$, contradicting \autoref{is_scrawl}.
\end{proof}

\begin{lem}\label{cir_c_scrawl}
 Let $M$ be a matroid and $\Ccal,\Dcal\se \Pcal(E)$
such that every $M$-circuit is a union of elements of $\Ccal$, every $M$-cocircuit is a union of elements of $\Dcal$ and $|C\cap D|\neq 1$ for every $C\in\Ccal$ and every $D\in\Dcal$.

Then $\Ccal(M)\se \Ccal\se \Scal(M)$ and  $\Ccal(M^*)\se \Dcal\se \Scal(M^*)$
\end{lem}

\begin{proof}
We begin by showing that $\Ccal(M) \se \Ccal$. For any circuit $o$ of $M$, pick an element $e$ of $o$. Since $o$ is a union of elements of $\Ccal$ there is $o' \in \Ccal$ with $e \in o' \se o$. Suppose for a contradiction that $o'$ isn't the whole of $o$, so that there is $f \in o \sm o'$. By \autoref{o_cap_b} there is some cocircuit $b$ of $M$ with $o' \cap b = \{e\}$. Then we can find $b' \in \Dcal$ with $e \in b' \se b$, and so $o' \cap b' = \{e\}$, giving the desired contradiction. Similarly we obtain that $\Ccal(M^*) \se \Dcal$.

The fact that $\Ccal \se \Scal(M)$ is immediate from \autoref{is_scrawl} since $\Ccal(M^*) \se \Dcal$, and the proof that $\Dcal \se \Scal(M^*)$ is similar.
\end{proof}

\section{What is a $\Psi$-matroid?}\label{sec:psi}

In this section we shall review the definitions of $\Psi$-circuits and $\Psi^{\complement}$-bonds for a graph $G$ with a specified set $\Psi$ of ends. Much of what we say will be a review of the early parts of \cite{DP:dualtrees}, though we shall work in a slightly more general context: in \cite{DP:dualtrees}, only finitely separable graphs are considered (a graph is {\em finitely separable} if any two vertices lie on opposite sides of some finite cut). We shall rely on\cite{DP:dualtrees} for the results we need about finitely separable graphs.

We say that two rays in a graph $G$ are {\em equivalent} if they cannot be separated by removing finitely many vertices from $G$. An {\em end} of $G$ is an equivalence class of rays under this relation, and the set of ends of $G$ is denoted $\Omega(G)$.

 Let $d$ be the distance function on $V(G) \sqcup (0, 1) \times E(G)$ considered as the ground set of the simplicial 1-complex formed from the vertices and edges of $G$. We define a topology {\scshape VTop} on the set $V(G) \sqcup \Omega(G) \sqcup (0,1) \times E(G)$ by taking basic open neighbourhoods as follows:
\begin{itemize}
\itum For $v \in V(G)$, the basic open neighbourhoods of $v$ are the $\epsilon$-balls $B_{\epsilon}(v) = \{x | d(v, x) < \epsilon\}$ for $\epsilon \leq 1$.
\itum For $(x, e) \in (0, 1) \times E$ we say $(x, e)$ is an {\em interior point} of $e$, and take the basic open neighbourhoods to be the $\epsilon$-balls about $(x, e)$ with $\epsilon \leq \min(x, 1-x)$.
\itum For $\omega \in \Omega(G)$, the basic open neighbourhoods of $\omega$ will be parametrised by the finite subsets $S$ of $V(G)$. Given such a subset, we let $C(S, \omega)$ be the unique component of $G - S$ that contains a ray from $\omega$, and let $\hat C(S, \omega)$ be the set of all vertices and inner points of edges contained in or incident with $C(S, \omega)$, and of all ends represented by a ray in $C(S, \omega)$. We take the basic open neighbourhoods of $\omega$ to be the sets $\hat C(S, \omega)$.
\end{itemize}

We call the topological space obtained in this way $|G|$. We will need a fundamental lemma about this topology. A {\em comb} in $G$ consists of a ray $R$ together with infinitely many vertex-disjoint finite paths having precisely their first vertex on $R$. $R$ is called the {\em spine} of the comb, and the final vertices of the paths are called the {\em teeth} of the comb. 

\begin{lem}\label{spine}
Let $G$ be a graph.
Let $X$ be a set of vertices of $G$ and $\omega$ an end of $G$.
Let $R_\omega$ be some ray in $\omega$.

Then $\omega$ is in the closure of $X$ if and only if there is a comb with spine $R_\omega$ all of whose teeth are in $X$.
\end{lem}

\begin{proof}
For the `if' direction, let $\hat C(S, \omega)$ be a basic open neighbourhood of $\omega$. Then only finitely many of the paths in the comb can meet $S$, so without loss of generality none of them do. Some tail of $R$ must lie in $C(S, \omega)$, so without loss of generality the whole of $R$ does. Then all teeth of the comb lie in $\hat C(S, \omega)$.
 
For the `only if' direction, we apply Menger's Theorem to get either 
infinitely many vertex-disjoint $R_\omega$-$X$-paths
or a finite vertex set $S$ whose removal separates $X$ from  $R_\omega$.
In the first case we are done and in the second we get a contradiction to the assumption that $\omega$ lies in the closure of $X$.
\end{proof}

For any set $\Psi$ of ends of $G$, we set $\Psi^{\complement} = \Omega(G) \setminus \Psi$ and $|G|_{\Psi} = |G| \setminus \Psi^{\complement}$. This topological space, derived from a graph, seems almost to fit the notion of graph-like space explored in \cite{BCC:graphic_matroids} (and closely related to the earlier work of \cite{ThomassenVellaContinua}). We can make this precise as follows:

\begin{dfn}
An \emph{almost graph-like space $G$} is a 
topological space (also denoted $G$) together with 
a \emph{vertex set} $V=V(G)$, an \emph{edge set} $E=E(G)$ and for each $e \in E$ a continuous map $\iota_e \colon [0, 1] \to G$ such that:
\begin{itemize}
\itum The underlying set of $G$ is $V\sqcup (0,1)\times E$
\itum For any $x \in (0, 1)$ we have $\iota_e(x) = (x, e)$.
\itum $\iota_e(0)$ and $\iota_e(1)$ are vertices (called the {\em endvertices} of $e$).
\itum $\iota_e \restric_{(0, 1)}$ is an open map.
\end{itemize} Such an almost graph-like space is a {\em graph-like space} if in addition for any $v, v' \in V$, there are disjoint open subsets $U, U'$ of $G$ partitioning $V(G)$ and with $v \in U$ and $v' \in U'$. This ensures that $V(G)$, considered as a subspace of $G$, is totally disconnected, and that $G$ is Hausdorff.
\end{dfn}

Thus we can give $|G|_{\Psi}$ the structure of an almost graph-like space, with edge set $E(G)$ and vertex set $V(G) \cup \Psi$.

Let $e$ be an edge in a graph-like space with $\iota_e(0) \neq \iota_e(1)$. Then $\iota_e$ is a continuous injective map from a compact to a Hausdorff space and so it is a homeomorphism onto its image. The image is compact and so is closed, and so is the closure of $(0, 1) \times \{e\}$ in $G$. So in this case $\iota_e$ is determined by the properties above and the topology of $G$. The same is true if $\iota_e(0) = \iota_e(1)$: in this case we can lift $\iota_e$ to a continuous map from $S^1 = [0, 1]/(0 = 1)$ to $G$, and argue as above that this map is a homeomorphism onto the closure of $(0, 1) \times \{e\}$ in $G$.

\begin{dfn}
We say that two vertices $v$ and $v'$ of an almost graph-like space $G$ are {\em equivalent} (denoted $v \sim v'$) if for any disjoint open subsets $U, U'$ of $G$ partitioning $V(G)$, $v$ and $v'$ lie on the same side of the partition. The {\em graph-like quotient} $\widetilde G$ of $G$ is the space obtained from $G$ by identifying equivalent vertices.
$\widetilde G$ has the structure of a graph-like space with the same edge set as $G$ and with vertex set $V(G) / \sim$.
\end{dfn}

\begin{lem}
If $G$ is an almost graph-like space, then $\widetilde G$ is a graph-like space.
\end{lem}
\begin{proof}
It is clear that $\widetilde{G}$ is an almost graph-like space. Let $[v]_{\sim}$ and $[v']_{\sim}$ be distinct vertices of $\widetilde G$. Then $v \not \sim v'$, and so there are disjoint open sets $U$ and $U'$ in $G$ which partition $V(G)$ and with $v \in U$ and $v' \in U'$. Then any pair of equivalent vertices of $G$ are either both in $U$ or both in $U'$, so $U$ and $U'$ induce disjoint open subsets $U/\sim$ and $U'/\sim$ of $\widetilde G$ which partition the vertices of $\widetilde G$ and such that $[v]_{\sim} \in U / \sim$ and $[v']_{\sim} \in U' / \sim$. 
\end{proof}

We say that a cut $b$ in a graph $G$ is {\em $\Psi$-bounded} if the closure of $b$ in $|G|_{\Psi}$ contains no ends. Thus if $b$ is $\Psi$-bounded and $\omega$ is an end in $\Psi$ then any ray to $\omega$ in $G$ lies eventually on one side of $b$ - we then say that $\omega$ is on that side of $b$.

\begin{lem}\label{vertices_equivalent}
Two vertices of $|G|_{\Psi}$ are equivalent if and only if they lie on the same side of every $\Psi$-bounded cut.
\end{lem}
\begin{proof}
For the `if' direction, let $v$ and $v'$ be inequivalent vertices of $|G|_{\Psi}$, and let $U$ and $U'$ be disjoint open subsets partitioning $V(|G|_{\Psi})$ with $v \in U$ and $v' \in U'$. Let $b$ be the cut of $G$ consisting of those edges with one endvertex in $U$ and the other in $U'$. We shall show that $b$ is $\Psi$-bounded. Let $\omega \in \Psi$. Without loss of generality $\omega \in U$ and so there is some $S$ with $\hat C(S, \omega) \subseteq U$. Let $e \in b$, so one endvertex is in $U'$. Then since $U'$ is open some interior point of $e$ is in $U'$, so that interior point of $e$ isn't in $\hat C(S, \omega)$, so $e$ doesn't meet $\hat C(S, \omega)$. Since $e$ was arbitrary, $\hat C(S, \omega) \cap b = \varnothing$ and so $\omega$ isn't in the closure of $b$, as required.

For the `only if' direction, let $v$ and $v'$ be equivalent vertices of $|G|_{\Psi}$ and let $b$ be a $\Psi$-bounded cut of $G$. For each end $\omega \in \Psi$ there is by the definition of $\Psi$-boundedness a basic open set $U_{\omega} = \hat C(S_{\omega}, \omega)$ that doesn't meet $b$. Each set $C(S_{\omega}, b)$ is connected and so lies entirely on one side of $b$. Letting the sides of $b$ be $X$ and $X'$, we may take $U = \bigcup_{v \in V(G) \cap X} B_{\frac12}(v) \cup \bigcup_{\omega \in \Psi \cap X} U_{\omega}$ and $U' = \bigcup_{v \in V(G) \cap X'} B_{\frac12}(v) \cup \bigcup_{\omega \in \Psi \cap X'} U_{\omega}$. Now since $v$ and $v'$ are equivalent they must either be both in $U$ or both in $U'$, so they lie on the same side of $b$. Since $b$ was arbitrary, we are done.
\end{proof}

For a vertex $v$ of $G$ and a ray $R$ of $G$, we say that $v$ {\em dominates} $R$ if there are infinitely many paths from $v$ to $R$, vertex-disjoint except at $v$. We say that $v$ {\em dominates} some end $\omega$ if it dominates some ray (or equivalently all rays) in $\omega$.

\begin{lem}
Let $v \in V(G)$ dominate some end $\omega \in \Psi$. Then $v$ and $\omega$ are equivalent as vertices of $|G|_{\Psi}$. 
\end{lem}
\begin{proof}
Let $R$ be a ray in $\omega$ and let $(P_i | i \in \Nbb)$ be a sequence of paths from $v$ to $\omega$ meeting only at $v$. Suppose for a contradiction that there is a $\Psi$-bounded cut with $v$ and $\omega$ on opposite sides. Then $R$ must eventually lie on the same side of $b$ as $\omega$, so without loss of generality it lies entirely on that side. For each $P_i$, let $v_i$ be the first vertex of $P_i$ on the same side of $b$ as $\omega$. Then $R$ together with the paths $v_iP_i$ forms a comb, so by \autoref{spine} the end $\omega$ is in the closure of the set of teeth $v_i$, so it is in the closure of $b$, which is the desired contradiction.
\end{proof}

We let $\simeq$ be the smallest equivalence relation identifying any vertex with any end that it dominates. If $G$ is finitely separable, then by [{\cite{DP:dualtrees}, {Lemma 6}}], no two vertices will be equivalent under $\simeq$. In \cite{DP:dualtrees}, the topological space $\widetilde G_{\Psi}$ is defined, for $G$ a finitely separable graph, to be the quotient of $|G|_{\Psi}$ by $\simeq$. By the above lemma, $\sim$ refines $\simeq$ and so there is a continuous quotient map $f_G \colon \widetilde G_{\Psi} \to \widetilde{|G|_{\Psi}}$.

\begin{lem}
If $G$ is finitely separable, then $f_G$ is an homeomorphism.
\end{lem}
\begin{proof}
Since $f$ is a quotient map, it suffices to show that it is injective.

Let $v$ and $v'$ be vertices of $\widetilde G_{\Psi}$. By [{\cite{DP:dualtrees}, {Lemma 6}}], there is a finite set $F$ of edges such that $v$ and $v'$ lie in disjoint open subsets of $\widetilde G_{\Psi} \setminus (0, 1) \times F$ whose union is $\widetilde G_{\Psi} \setminus (0, 1) \times F$. Let $C$ be the connected component of $G \setminus F$ containing $v$ (or a ray to $v$ if $v$ is an end), and let $b \subseteq F$ be the cut consisting of edges with one endvertex in $C$ and the other not. Since $b$ is finite, it is a $\Psi$-bounded cut, and so $v \not \sim v'$, as required.
\end{proof}

We therefore extend the definition in \cite{DP:dualtrees} by taking $\widetilde G_{\Psi}$ for $G$ an arbitrary graph to be the graph-like quotient of $|G|_{\Psi}$.

In \cite{BCC:graphic_matroids}, topological circuits and topological bonds are defined in any graph-like space.
A \emph{circuit of $\widetilde G_{\Psi}$}, or just \emph{a $\Psi$-circuit,} is an edge set whose
$\widetilde G_{\Psi}$-closure is homeomorphic to the unit circle.
A \emph{bond of $\widetilde G_{\Psi}$}, or just \emph{a $\Psi\ct$-bond}-is an edge set 
of a  minimal nonempty $\Psi$-bounded cut.
In the following sense the $\Psi$-circuits and $\Psi\ct$-bonds behave like the circuits and cocircuits of some matroid.

\begin{lem}\label{never_meet_just_once}
 No $\Psi$-circuit meets any $\Psi\ct$-bond in a single edge.
\end{lem}

\begin{proof}
Suppose for a contradiction that some $\Psi$-circuit $o$ meets some $\Psi\ct$-bond $b$ in a single edge $f$

Then $\widetilde G_\Psi$ with all the interior points of edges of $b$ removed has two connected components, namely the two sides of the bond.  
This contradicts the fact that $o-f$ is connected and contains both endvertices of $f$.
\end{proof}

We say that \emph{$(G,\Psi)$ induces a matroid $M$} if $E(M)=E(G)$ and the $M$-circuits are the $\Psi$-circuits and the $M$-cocircuits are the $\Psi\ct$-bonds. In this case, we call $M$ the {\em $\Psi$-matroid} of $G$. Even if we don't get a matroid, we call $(\Ccal, \Dcal)$, where $\Ccal$ is the set of $\Psi$-circuits of $G$ and $\Dcal$ is the set of $\Psi\ct$-bonds of $G$, the {\em $\Psi$-system} of $G$.

A $\Psi$-tree is an edge set maximal with the property that it includes no $\Psi$-circuit.
The main results of Diestel and Pott \cite{DP:dualtrees} are phrased in terms of $\Psi$-trees.
These results let them to suspect that the $\Psi$-trees are the bases of some matroid.
Although we shall mostly work with $\Psi$-circuits and $\Psi\ct$-cocircuits instead,
the fact that our results do confirm this suspicion in many cases follows from the following lemma.

\begin{lem}\label{psi_trees}
If $(G,\Psi)$ induces a matroid, then the bases of this matroid are the $\Psi$-trees.
\qed
\end{lem}

We say that \emph{$G$ and $G^*$ are plane duals}
if there is an isomorphism $\iota$ from $E(G)$ to $E(G^*)$ that maps the $G$-cycles
to the $G^*$-bonds.
In \cite{MR2575097}, it is proved that $\iota$ induces a bijection $\iota_\Omega$ between the ends of $G$ and the ends of $G^*$.
Then \autoref{psi_trees} yields the following.

\begin{cor}\label{embeddings}
Let $G$ and $G^*$ be two finitely separable graphs that are plane duals, as witnessed by some map $\iota$.
If $(G,\Psi)$ induces a matroid, then its dual is induced by $(G^*,\iota_\Omega(\Psi\ct))$.
\qed
\end{cor}

\section{Orthogonality axioms}\label{orth}

The purpose of this section is to axiomatize countable matroids in a new way using 
more axioms, each of which is simpler to check.
This makes the process of testing whether a given system is a matroid
more straightforward.
We were motivated by Section 2.2 from \cite{BC:rep_matroids}.


The {\em orthogonality axioms} are as follows, where we think of $\Ccal$ as the set of circuits of the matroid
and $\Dcal$ the set of cocircuits.

\begin{itemize}
	\item[(C1)] $\varnothing\notin \Ccal$
	\item[(C2)] No element of $\Ccal$ is a subset of another.  
        \item[(C1$^*$)] $\varnothing\notin \Dcal$
        \item[(C2$^*$)] No element of $\Dcal$ is a subset of another.  
        \item[(O1)] $|C\cap D|\neq 1$ for all $C\in \Ccal$ and $D\in \Dcal$. 
        \item[(O2)] For all partitions $E=P_{Co}\dot\cup P_{De}\dot\cup \{e\}$
either $P_{Co}+e$ includes an element of $\Ccal$ through $e$ or
$P_{De}+e$ includes an element of $\Dcal$ through $e$.
\item[(O3)]For every $C\in \Ccal$, $e\in C$ and $X\subseteq E$, there is some $C_{min}\in \Ccal$ with $e\in C_{min} \se X \cup C$ such that $C_{min}\sm X$ is minimal.
\item [(O3$^*$)]For every $D\in \Dcal$, $e\in D$ and $X\subseteq E$, there is some $D_{min}\in \Dcal$ with $e\in D_{min} \se X \cup D$ such that $D_{min}\sm X$ is minimal.
\end{itemize}

The aim of the rest of this section will be to show the following:

\begin{thm}\label{ortho_axioms}
Let $E$ be a countable set and let $\Ccal,\Dcal\subseteq \Pcal(E)$.

Then $\Ccal$ is the set of circuits of a matroid and $\Dcal$ is the set of cocircuits of the same matroid
if and only if $\Ccal$ and $\Dcal$ satisfy the orthogonality axioms.
\end{thm}

To determine, rather than define, a matroid, the last four of the orthogonality axioms suffice. What we mean by this slightly subtle distinction is captured by the following strengthening of the theorem above: 

\begin{thm}\label{ortho_axioms+}
Let $E$ be a countable set and let $\Ccal,\Dcal\subseteq \Pcal(E)$.

Then there is a matroid $M$ such that $\Ccal(M)\subseteq \Ccal \subseteq \Scal(M)$
and $\Ccal(M^*)\subseteq \Dcal \subseteq \Scal(M^*)$ if and only if 
 $\Ccal$ and $\Dcal$ satisfy the last four orthogonality axioms.
\end{thm}

First we show the following lemma.
For a set $\Ccal\subseteq \Pcal(E)$, let $\Ccal^\perp$
be the set of those subsets of $E$ that meet no element of $\Ccal$ just once.
Note that $(O1)$ is equivalent to $\Dcal$ being a subset of $\Ccal^\perp$.

\begin{lem}\label{cireli_partitioning}
Let $\Ccal \subseteq \Pcal(E)$.
Then $\Ccal$ and $\Ccal^\perp$ satisfy $(O2)$ if and only if
$\Ccal$ satisfies circuit elimination $(C3)$. 
\end{lem}

\begin{proof}
For the ``if'' implication, suppose we are given a partition  
 $E=P_{Co}\dot\cup P_{De}\dot\cup \{e\}$ where $e\in E$ such that 
 $P_{Co}+e$ does not include an element of $\Ccal$ through $e$.

Let $D$ consist of those elements $x$ of $P_{De}+e$ such
that $P_{Co}+x$ does not include an element of $\Ccal$ through $x$.  

Suppose for a contradiction that $D\notin \Ccal^\perp$.
Then there is some $C\in \Ccal$ meeting $D$ only in a single element $z$.
Let $X=C\cap ((P_{De}+e)\sm D)$. For any $x\in X$ pick an element $C_x$ of $\Ccal$
such that $C_x\se P_{Co}+x$ and $x\in C_x$.
Applying circuit elimination to $z,C,X$ and the $C_x$ yields an element of $\Ccal$
meeting $P_{De}+e$ exactly in $z$, which contradicts the choice of $z$.

It remains to show the ``only if''-implication.
Suppose we are given $z,C,X$ and the $C_x$ as in the circuit elimination axiom.
Put $e=z$, and $P_{Co}=(C\cup \bigcup_{x\in X}C_x)\sm (X+z)$, and $P_{De}=(E\sm P_{Co})-z$.

To prove circuit elimination, it remains to show that there is no
element  $D\subseteq P_{De}+z$ of $\Ccal^\perp$ through $z$.
Since $z\in C$ and $C\cap D\subseteq X+z$, any such set $D$ would contain some $x\in X$ since $D\in \Ccal^\perp$.
But then $C_x\cap D=\{x\}$, which is impossible since $D\in \Ccal^\perp$.
This completes the proof.
\end{proof}

\begin{proof}[Proof of \autoref{ortho_axioms+}.]
First we show the ``only if''-implication.
The axiom $(O1)$ follows from \autoref{is_scrawl}. 
To show $(O2)$ consider the matroid $M_e$ on $\{e\}$ obtained from $M$ by contracting $P_{Co}$ and deleting $P_{De}$. If $M_e$ is a loop, then by \autoref{rest_cir} $P_{Co}+e$ includes a circuit through $e$, and
if $M_e$ is a co-loop, then by the dual of \autoref{rest_cir} $P_{De}+e$ includes a cocircuit through $e$.

By duality, it remains to show $(O3)$.
For this we consider the matroid $M_X$ obtained from $M$ by contracting $X-e$.
Note that $(C\sm X)+e$ is an $M_X$-scrawl.
Hence we may pick any $M_X$-circuit through $e$ included in $(C\sm X)+e$.
By \autoref{rest_cir}, this circuit extends to an $M$-circuit $C_{min}$, which has the desired properties.
This completes the proof of the ``only if''-implication.

For the ``if''-implication, our aim is to show that the set 
$\Ccal_{min}$ of minimal non-empty elements of $\Ccal$ is the set of circuits of a matroid $M$.
Note that circuit elimination $(C3)$ for $\Ccal$ follows from \autoref{cireli_partitioning},
and this implies circuit elimination for $\Ccal_{min}$ using $(O3)$.

Next, we prove $(CM)$ for $\Ccal_{min}$.
Suppose we are given a set $I$ not including an element of $\Ccal_{min}$
and a set $X$ with  $I\se X\se E$.
Put $I_0=I$ and $J_0=E\sm X$.

Let $e_1,e_2,\ldots$ be an enumeration of $X$.
We shall construct a partition of $E$ into $I_\infty$ and $J_\infty$
such that $I_\infty$ is a base of $X$.
The construction will be recursive. So we take 
$I_\infty=\bigcup_{n\in \Nbb} I_n$ and $J_\infty=\bigcup_{n\in \Nbb} J_n$
where we construct the $I_n$ and $J_n$ both at step $n$ to satisfy the following conditions.
\begin{enumerate}
 \item $I_n$ and $J_n$ are disjoint.
\item $I_j\se I_n$ for all $j\leq n$.
\item $J_j\se J_n$ for all $j\leq n$.
\item $e_n\in I_n\cup J_n$.
\item If $e_n\in I_n$, then there is some $D\in \Dcal$ with $D\se J_n+e_n$ through $e_n$.
\item If $e_n\in J_n$, then there is some $C\in \Ccal$ with $C\se I_n+e_n$ through $e_n$.
\item If $J_n$ includes any $D\in \Dcal$, then $D\se J_0$.
\item If $I_n$ includes any $C\in \Ccal$, then $C\se I_0$. (That is, $C=\varnothing$: this condition says that $I_n$ is independent.)
\end{enumerate}

What we do at step $n$ depends on whether there is any $C \in \Ccal$ with $e_n \in C \subseteq I_{n-1} + e_n$. If there is such a $C$, we let $I_n = I_{n-1}$ and $J_n = J_{n-1} + e_n$. The only nontrivial condition in this case is $(7)$. By the induction hypothesis, any $D$ violating this condition would contain $e_n$ and so would meet $C$ just once, contradicting $(O1)$.

Next, we consider the case that $e_n \notin J_{n-1}$. In this case, we let $I_n = I_{n-1} + e_n$, but the construction of $J_n$ is more complex. First, we note that by $(O2)$ applied to $e_n$, $I_{n-1}$ and $E \sm I_{n-1} - e_n$ we can obtain some $D \in \Dcal$ with $e_n \in D \se E \sm I_{n-1}$. Then using $(O3^*)$ we may assume that $D$ is chosen with these properties so that $D \sm J_{n-1}$ is minimal. We take $J_n = (J_{n-1} \cup D) - e_n$.

Once more, the only nontrivial condition is $(7)$. Suppose for a contradiction that there is some $D'$ violating this condition. Then $D'$ must meet $D \sm J_{n-1}$ in some element $x$. We showed above that $\Ccal$ satisfies $(C3)$, and by symmetry we may also show that $\Dcal$ satisfies $(C3)$. We apply this with $X = \{x\}$ to $D$ and $D'$ to obtain $D'' \in \Dcal$ with $e_n \in D'' \se (D \cup J_{n-1}) - x$, contradicting the minimality of $D\sm J_{n-1}$.

The remaining case is that $e_n \in J_{n-1}$. In this case, we let $J_n = J_{n-1}$ and, dualising the construction from the last case, we choose $C \in \Ccal$ such that $e_n \in C \subseteq E \sm (J_{n-1}-e_n)$ and $C \sm I_{n-1}$ is minimal subject to these conditions. This construction succeeds for a reason dual to that given in the last case.

This completes the recursive construction. As promised, we take $I_{\infty}=\bigcup_{n\in \Nbb} I_n$ and $J_\infty=\bigcup_{n\in \Nbb} J_n$. It is clear that this is a partition of $E$. Next, we show that $I_\infty$ includes no element $C$ of $\Ccal_{min}$. Suppose for a contradiction that there is such a $C$. Then there is some $n$ with $e_n$ in $C$. Then by $(5)$ there is some $D \in \Dcal$ with $e_n \in D \subseteq J_n + e_n \subseteq J_\infty + e_n$, so that $C \cap D = \{e_n\}$, violating $(O1)$.

We can also show that $I_{\infty}$ is maximal amongst the independent subsets of $X$. Suppose for a contradiction that there is a bigger independent set $I'$, and pick some $n$ with $e_n \in I' \sm I$. Then by $(6)$ there is $C \in \Ccal$ with $e_n \in C \se I_n + e_n \se I'$, contradicting the independence of $I'$ as by (O3), $C$ is a union of elements of $C_{min}$. This completes the proof that $\Ccal_{min}$ is the set of circuits of some matroid $M$.

By $(O3)$, every element of $\Ccal$ is a union of elements of $\Ccal(M)$.
Hence $\Ccal(M)\subseteq \Ccal \subseteq \Scal(M)$. $(O1)$ and 
\autoref{is_scrawl} imply that $\Dcal \subseteq \Scal(M^*)$. It remains only to show that $\Ccal(M^*) \se \Dcal$. So let $D$ be any cocircuit of $M$. Let $e \in D$, and apply $(O2)$ to $e$, $E \sm D$ and $D-e$. There can't be $C \in \Ccal$ with $e \in C \se (E \sm D) + e$, as then we would have $C \cap D = \{e\}$, which is impossible with $C$ a scrawl and $D$ a cocircuit. So there is some $D' \in \Dcal$ with $e \in D' \se D$, and we must have $D' = D$ since no nonempty proper subset of $D$ can be a scrawl of $M^*$.
\end{proof}

We are left with the open questions of whether the restriction that $E$ should be countable can be removed from Theorems 0.1 and 0.2, or if not whether there is a simple axiom which can be added to fix this defect.

We can also show that for tame matroids we do not need $(O3)$ or $(O3)^*$. More precisely:

\begin{thm}
Let $\Ccal$ and $\Dcal$ be sets of subsets of a countable set $E$ satisfying $(O1)$ and $(O2)$, and such that for any $C \in \Ccal$ and any $D \in \Dcal$ the intersection $C \cap D$ is finite. Then $\Ccal$ and $\Dcal$ also satisfy $(O3)$ and $(O3)^*$, so that they induce a matroid as above.
\end{thm}
\begin{proof}
By symmetry, it is enough to show $(O3)$. Let $C \in \Ccal$, $e \in C$ and $X \se E$. Let $\Ycal$ be the set of subsets $Y$ of $C \sm X$ such that $e \in Y$ and for every $D \in \Dcal$ with $D \cap X = \varnothing$ we have $|Y \cap D| \neq 1$. 
We will use Zorn's Lemma to show that $\Ycal$ has a minimal element. $\Ycal$ is nonempty because it contains $C \sm X$ by $(O1)$. Let $\Zcal$ be a nonempty chain of elements of $\Ycal$. We shall show that $\bigcap \Zcal$ is in $\Ycal$ and so forms a lower bound for $\Zcal$ there. Evidently $e \in \bigcap \Zcal$. For any $D \in \Dcal$ with $D \cap X = \varnothing$ we know that $D \cap (C \sm X)$ is finite and so we can find a finite subset $\Zcal'$ of $\Zcal$ such that for any $f \in D \cap C \sm \bigcap \Zcal$ there is $Z \in \Zcal'$ such that $f \not \in Z$. Let $Z$ be the least element of $\Zcal'$. Then $|\bigcap \Zcal \cap D| = |Z \cap D| \not = 1$.

Let $Y$ be a minimal element of $\Ycal$. We apply $(O2)$ to the partition $E = (X \cup Y - e) \dot \cup (E \sm X \sm Y) \dot \cup \{e\}$. By the definition of $\Ycal$ there is no $D \in \Dcal$ with $e \in D \subseteq E \sm X \sm Y$, so there is some $C_{min} \in \Ccal$ with $e \in C_{min} \subseteq X \cup Y$. For any other $C' \in \Ccal$ with $e \in C' \subseteq X \cup C$, we have $C' \sm X \in \Ycal$ by $(O1)$ and so $C_{min} \sm X \se C' \sm X$.
\end{proof}

In the following, we look at how tameness, $(O1)$ and $(O2)$ look for $\Ccal$ the set of $\Psi$-circuits and $\Dcal$ the set of $\Psi\ct$-bonds for a locally finite graph $G$ and $\Psi\se \Omega(G)$. We abbreviate $G_\Psi=|G|\sm (\Omega(G)\sm \Psi)$.

First, we look at tameness. If a $\Psi$-circuit $o$ and a $\Psi\ct$-bond $b$ meet infinitely, this gives rise to a minimal cover of $o$ with infinitely many open sets, contradicting the compactness of $o$. 
Hence tameness is implied by the fact that every circuit is compact.

Next, we look at $(O1)$.
If we have a $\Psi$-circuit $o$, then for every $e\in o$, the closure of $o- e$ in $G_\Psi$
is still connected and hence there cannot be a $\Psi\ct$-bond $b$ meeting $o$ only in $e$.
Thus $(O1)$ is implied by the fact that for every $\Psi$-circuit $o$ and every $e\in o$, the closure of $o- e$ in $G_\Psi$
is connected.

Finally, we look at $(O2)$, so we are 
given a partition $E=P_{Co}\dot\cup P_{De}\dot\cup \{e\}$.
Let $\bar P_{Co}$ be the closure of the edge set $P_{Co}$ in $G_\Psi$.
Let us consider the topological space $G_\Psi\cap \bar P_{Co}$.
 Then $(O2)$ says that we either find an arc joining the two endvertices of $e$ or we find a
$\Psi{\ct}$-bond whose induced topological bond separates the two endvertices.
The first is equivalent to the statement that the two endvertices of $e$ are in the same arc-component since the topological circles in $|G|_\Psi$ are precisely the $\Psi$-circuits.

The second is equivalent to the statement that the two endvertices of $e$ are 
not in the same connected component. Indeed, if there is a such a bond, then the two endvertices are clearly not in the same connected component.
For the converse, we assume that there is an open partition of $|G|_\Psi$ into two sets $C_1$ and $C_2$ with the two endvertices of $e$ on different sides.
Let $V_1$ be the set of those vertices in $C_1$.
Then the set of edges crossing the $G$-separation $(V_1,V(G)\sm V_1)$
is a (possibly infinite) cut of $G$.
This cut is a disjoint
union of bonds, which are all $\Psi\ct$-bonds. From these, the bond including $e$ is the desired one. 

Hence $(O2)$ is equivalent to the following:
The two endvertices of $e$ lie in the same connected component of $G_\Psi\cap \bar P_{Co}$ if and only if 
they lie in the same arc-component of $G_\Psi\cap \bar P_{Co}$.

The question whether any connected subspace of $|G|$ is path-connected was solved by Georgakopoulos in \cite{cnpc}. Idneed, he constructed a locally finite graph $G$
such that $|G|$ has a subset $S$ that is connected but not path-connected. Note that since $|G|$ is a Hausdorff space, path-connectedness is equivalent to arc-connectedness.
It is straightforward to show that $(G,S\cap\Omega(G))$ does not induce a matroid since it does not satisfy $(O2)$ for $P_{Co}=S\cap E(G)$. 
We note for future reference that Georgakopoulos's argument heavily relies on the Axiom of Choice.
The purpose of this paper is to examine for which $G$ and $\Psi$, the pair $(G,\Psi)$
induces a matroid.

\section{Trees of matroids I}\label{treesofmatroids}

We wish to paste together infinite collections of matroids to obtain interesting new infinite matroids. Before we can be more explicit about this construction, we must give a precise account of the configurations of matroids we will seek to paste together. These will be given by tree-like structures.

\begin{dfn}
A {\em tree $\Tcal$ of matroids} consists of a tree $T$, together with a function $M$ assigning to each node $t$ of $T$ a matroid $M(t)$ on ground set $E(t)$, such that for any two nodes $t$ and $t'$ of $T$, if $E(t) \cap E(t')$ is nonempty then $tt'$ is an edge of $T$.

For any edge $tt'$ of $T$ we set $E(tt') = E(t) \cap E(t')$. We also define the {\em ground set} of $\Tcal$ to be $E = E(\Tcal) = \left(\bigcup_{t \in V(T)} E(t)\right) \setminus \left(\bigcup_{tt' \in E(T)} E(tt')\right)$. 

We shall refer to the edges which appear in some $E(t)$ but not in $E$ as {\em dummy edges} of $M(t)$: thus the set of such dummy edges is $\bigcup_{tt' \in E(T)} E(tt')$.
\end{dfn}

The idea is that the dummy edges are to be used only to give information about how the matroids are to be pasted together, but they will not be present in the final pasted matroid, which will have ground set $E(\Tcal)$. 

We shall consider a couple of different sorts of pasting. First, in this section, we will consider a type of pasting corresponding to 2-sums. Later, in \autoref{tm2}, we will define a type of pasting along larger separators. In each case, we will make use of some additional information to control the behaviour at infinity: a set $\Psi$ of ends of $T$. The first type of pasting is only possible for a restricted class of trees of matroids.

\begin{dfn}
A tree $\Tcal = (T, M)$ of matroids is {\em of overlap 1} if, for every edge $tt'$ of $T$, $|E(tt')| = 1$. In this case, we denote the unique element of $E(tt')$ by $e(tt')$.

Given a tree of matroids of overlap 1 as above and a set $\Psi$ of ends of $T$, a {\em $\Psi$-pre-circuit} of $\Tcal$ consists of a connected subtree $C$ of $T$ together with a function $o$ assigning to each vertex $t$ of $C$ a circuit of $M(t)$, such that all ends of $C$ are in $\Psi$ and for any vertex $t$ of $C$ and any vertex $t'$ adjacent to $t$ in $T$, $e(tt') \in o(t)$ if and only if $t' \in C$. The set of $\Psi$-pre-circuits is denoted $\overline\Ccal(\Tcal, \Psi)$. 

Any $\Psi$-pre-circuit $(C, o)$ has an {\em underlying set} $\underline{(C, o)} = E \cap \bigcup_{t \in V(C)} o(t)$. Nonempty subsets of $E$ arising in this way are called {\em $\Psi$-circuits} of $\Tcal$. The set of $\Psi$-circuits of $\Tcal$ is denoted $\Ccal(\Tcal, \Psi)$.
\end{dfn}

We shall show in \autoref{sec:2sums} that $\Ccal(\Tcal, \Psi)$ very often gives the set of circuits of a matroid on $E_{\Tcal}$. To do this, we will make use of the orthogonality axioms, and so we will also need a specified collection of putative cocircuits. These will be given by the $\Psi\ct$-circuits of a tree of matroids dual to $\Tcal$. Not only is there a natural notion of duality for trees of matroids, there are also natural notions of contraction and deletion.

\begin{dfn}
Let $\Tcal = (T, M)$ be a tree of matroids. Then the {\em dual} $\Tcal^*$ of $\Tcal$ is given by $(T, M^*)$, where $M^*$ is the function sending $t$ to $(M(t))^*$. For a subset $C$ of the ground set, the tree of matroids $\Tcal/C$ obtained from  $\Tcal$ by {\em contracting} $C$ is given by $(T, M/C)$, where $M/C$ is the function sending $t$ to $M(t)/(C \cap E(t))$. For a subset $D$ of the ground set, the tree of matroids $\Tcal\backslash D$ obtained from  $\Tcal$ by {\em deleting} $D$ is given by $(T, M \backslash D)$, where $M \backslash D$ is the function sending $t$ to $M(t) \backslash (C \cap E(t))$. 
We say that a tree of matroids $\Tcal$ of overlap 1 together with a set $\Psi$ of its ends {\em induce a matroid} $M = M(\Tcal, \Psi)$ if $\Ccal(M) \subseteq \Ccal(\Tcal, \Psi) \subseteq \Scal(M)$ and $\Ccal(M^*) \subseteq \Ccal(\Tcal^*, \Psi\ct) \subseteq \Scal(M^*)$.
\end{dfn}

\begin{lem}
For any tree $\Tcal$ of matroids, $\Tcal = \Tcal^{**}$. For any disjoint subsets $C$ and $D$ of the ground set of $\Tcal$ we have $(\Tcal / C)^* = \Tcal^* \backslash C$, $(\Tcal \backslash D)^* = \Tcal^* / D$ and $\Tcal / C \backslash D = \Tcal \backslash D /C$. If $\Tcal$ has overlap 1 and $(\Tcal, \Psi)$ induces a matroid $M$, then $(\Tcal/C\backslash D, \Psi)$ induces the matroid $M /C \backslash D$ and $(\Tcal^*, \Psi\ct)$ induces the matroid $M^*$. \qed
\end{lem}

We will sometimes use the expression {\em $\Psi\ct$-cocircuits of $\Tcal$} for the $\Psi\ct$-circuits of $\Tcal^*$.

We will examine the question of when $(\Tcal, \Psi)$ induces a matroid using the orthogonality axioms. The question of whether $(O2)$ holds for these systems is tricky and will be addressed in \autoref{sec:2sums}. However, we are already in a position to give simple proofs of $(O1)$, and of tameness if all the $M(t)$ are tame.

\begin{lem}[$(O1)$ for trees of matroids of overlap 1]\label{2SumsO1}
Let $\Tcal = (T, M)$ be a tree of matroids, $\Psi$ a set of ends of $T$, and let $(C, o)$ and $(D, b)$ be respectively a $\Psi$-pre-circuit of $\Tcal$ and a $\Psi\ct$-pre-circuit of $\Tcal^*$. Then $|\underline{(C, o)} \cap \underline{(D, b)}| \neq 1$.
\end{lem}
\begin{proof}
Suppose for a contradiction that $|\underline{(C, o)} \cap \underline{(D, b)}| = \{e\}$, with $e \in t_0$. We recursively construct a sequence of nodes $t_n \in C \cap D$ forming a ray from $t_0$. To construct $t_n$, we note that $o(t_{n-1})$ meets $b(t_{n-1})$ (in $e$ if $n = 1$, and in $e(t_{n-2}t_{n-1})$ if $n>1$), so since they are respectively a circuit and a cocircuit of $M(t_{n-1})$ they must meet at least twice. Since they cannot meet in any edge of $E$, they must meet in some edge $e(t_{n-1}t_n)$ with $t_n$ adjacent to $t_{n-1}$ in $T$ and $t_n \neq t_{n-2}$ (for $n>1$). It follows that $t_n \in C \cap D$. Then the end of this ray is in $\Psi$ by the definition of $(C, o)$ and is in $\Psi^{\complement}$ by the definition of $(D, b)$, which is the desired contradiction.
\end{proof}

\begin{lem}[Tameness for trees of tame matroids of overlap 1]\label{2Sumstame}
Let $\Tcal = (T, M)$ be a tree of tame matroids, $\Psi$ a set of ends of $T$, and let $(C, o)$ and $(D, b)$ be respectively a $\Psi$-pre-circuit of $\Tcal$ and a $\Psi\ct$-pre-circuit of $\Tcal^*$. Then $\underline{(C, o)} \cap \underline{(D, b)}$ is finite.
\end{lem}
\begin{proof}
Otherwise $C \cap D$ is infinite, and is locally finite since the $M(t)$ are all tame, and so it has an end $\omega$ of $T$ in its closure. Then $\omega$ is in $\Psi$ by the definition of $(C, o)$ and is in $\Psi^{\complement}$ by the definition of $(D, b)$, which is a contradiction.
\end{proof}

We shall later show that any $\Psi$-system for a locally finite graph can be recovered by a more complex version of the construction above from a tree of finite matroids. We illustrate this by showing that many interesting $\Psi$-systems can already be recovered from the construction given above.

\begin{dfn}
Let $G$ be a graph. A {\em tree structure} on $G$ is a tree $T$ whose nodes form a partition of the vertices of $G$, such that distinct nodes are adjacent in $T$ if and only if they contain adjacent vertices of $G$ and the induced subgraph on each partition class is finite and connected. A tree structure {\em has width 2} if and only if for any pair of adjacent partition classes in $T$ there are precisely 2 edges of $G$ with one endvertex in each class.
\end{dfn}

\begin{rem}
Any tree structure $T$ on $G$ induces a tree decomposition of $G$, in which the parts are the sets $E(t, t')$ of edges of $G$ with one endvertex in $t$ and the other in $t'$, for $t$ and $t'$ (not necessarily distinct) nodes of $T$.
\end{rem}

\begin{eg}\label{eg:wildstruc}
The wild cycle graph (so called because it includes a wild cycle in the sense of \cite{DiestelBook10}), depicted in \autoref{fig:wildcyc}, has a tree structure of width 2. The grey blobs represent the nodes of the tree.
\end{eg}

\begin{figure}
\begin{center}
 \includegraphics[height=6cm]{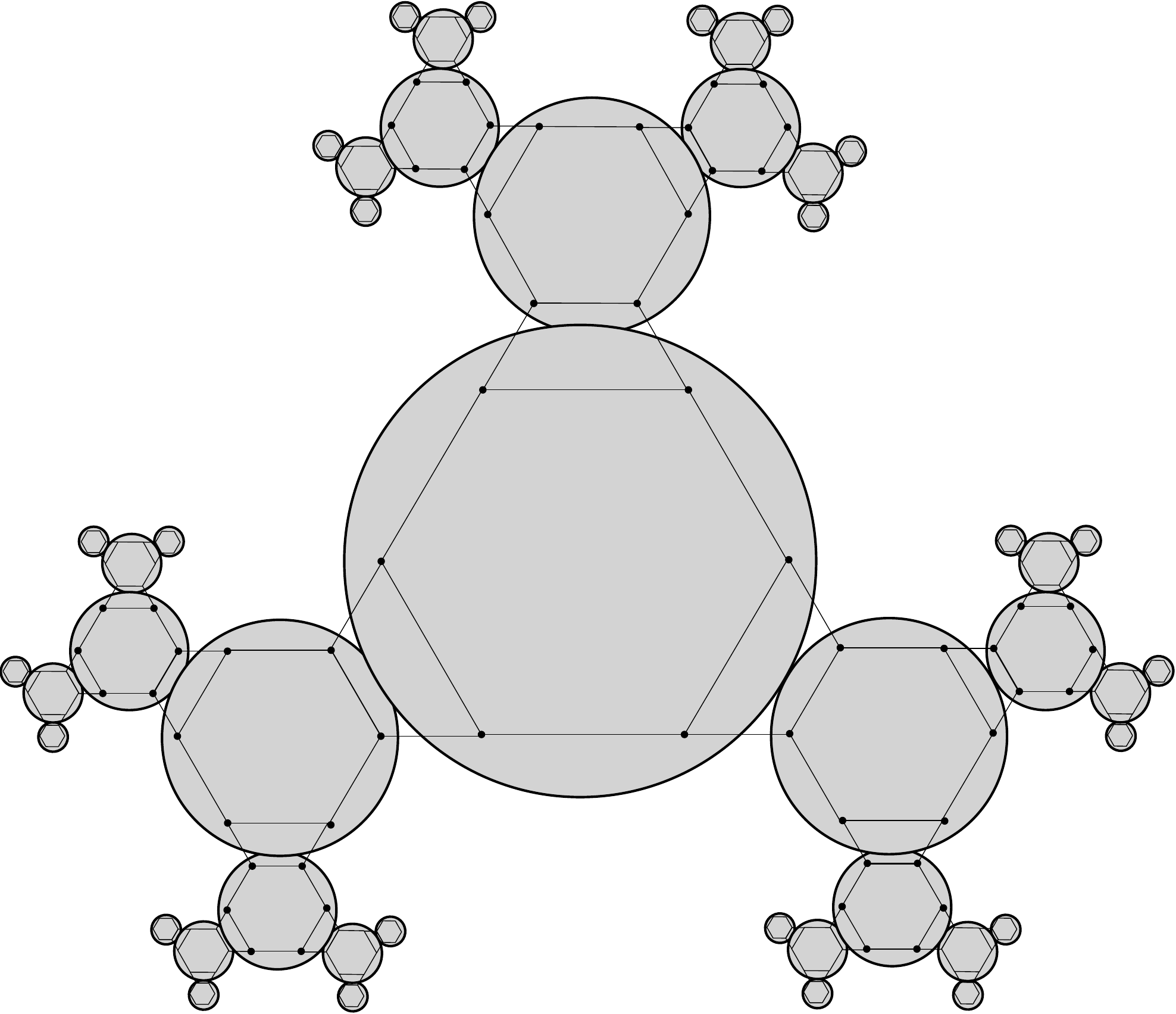}
\end{center}
\caption{A tree structure on the Wild Cycle Graph}\label{fig:wildcyc}
\end{figure}

\begin{lem}
Let $T$ be a tree structure on a locally finite graph $G$. Then there is a canonical homeomorphism from the ends of $G$ to the ends of $T$, sending an end $\omega$ of $G$ to the unique end in the closure of the set of vertices of $T$ that meet some ray $R$ to $\omega$.
\end{lem}
\begin{rem}
We shall use this homeomorphism to identify the ends of $T$ with those of $G$.
\end{rem}
\begin{proof}
First, we show that for any ray $R$ in $G$ there is a unique end $\varphi(R)$ of $T$ in the closure of the set of vertices of $T$ that meet $R$. There is certainly at least one such end, since $R$ is infinite and so must meet infinitely many of the (finite) partition classes. If there were 2, say $\omega$ and $\omega'$, then for any vertex $t$ of $T$ whose removal separated $\omega$ and $\omega'$, $R$ would have to meet $t$ infinitely often, which would be a contradiction.

A similar argument shows that $\varphi(R)$ only depends on the end of $G$ containing $R$: if there were 2 equivalent rays $R$ and $R'$ in $G$, with $\varphi(R) \neq \varphi(R')$, then for any vertex $t$ of $T$ whose removal separated $\varphi(R)$ from $\varphi(R')$, $R$ and $R'$ would eventually have to lie in the same component of $G \setminus t$, and so the components of $T - t$ meeting $R$ and $R'$ infinitely often would be the same, which would be a contradiction.

Thus $\varphi$ induces a map $\tilde \varphi$ taking ends of $G$ to ends of $T$. This map is injective, because for any distinct ends $\omega$ and $\omega'$ of $G$ there is a finite set $X$ of vertices of $G$ separating $\omega$ from $\omega'$ in $G$: the (finite) set of vertices of $T$ containing elements of $X$ then separates $\tilde \varphi(\omega)$ from $\tilde \varphi(\omega')$ in $T$. It is surjective, because for any ray $R$ in $T$ there is a ray in $G$ meeting exactly the nodes of $T$ on $R$ (here we use the fact that each node $t$ of $T$ is connected in $G$). It is continuous because for any node $t$ of $T$ the components of $G \setminus t$ are precisely the unions of the components of $T - t$, and it is open by the same fact together with the fact that any finite set $X$ of vertices of $G$ is a subset of a union of finitely many nodes of $T$.
\end{proof}

\begin{dfn}
Given a graph $G$ together with a tree structure $T$ on $G$ and a node $t$ of $T$, the {\em torso} $\tau(t)$ of $G$ at a node $t$ is the graph constructed as follows: the vertices are the elements of $t$, together with a new dummy vertex $v_e$ for each edge $e$ of $G$ with one endpoint in $t$ and the other not in $t$. The edges are of three types: edges of $G$ with both ends in $t$, an edge $vv_e$ for each edge $e = vv'$ of $G$ with $v \in t$ and $v' \in t'$ with $t'$ adjacent to $t$, and an edge joining any two dummy vertices corresponding to edges of $G$ from vertices in $t$ to vertices in the same adjacent node $t'$ of $T$.

For a graph $G$ with a tree structure $T$ this gives a corresponding tree of finite matroids $\Tcal(G, T) = (T, t \mapsto M(\tau(t)))$. 
\end{dfn}

Observe that if $T$ has width 2, then $\Tcal(G, T)$ has overlap 1.

\begin{eg}\label{eg:wildtorso}
Each torso arising from the tree structure in \autoref{eg:wildstruc} is isomorphic to the graph in \autoref{fig:wildtorso}.
\end{eg}

\begin{figure}
\begin{center}
 \includegraphics[height=4cm]{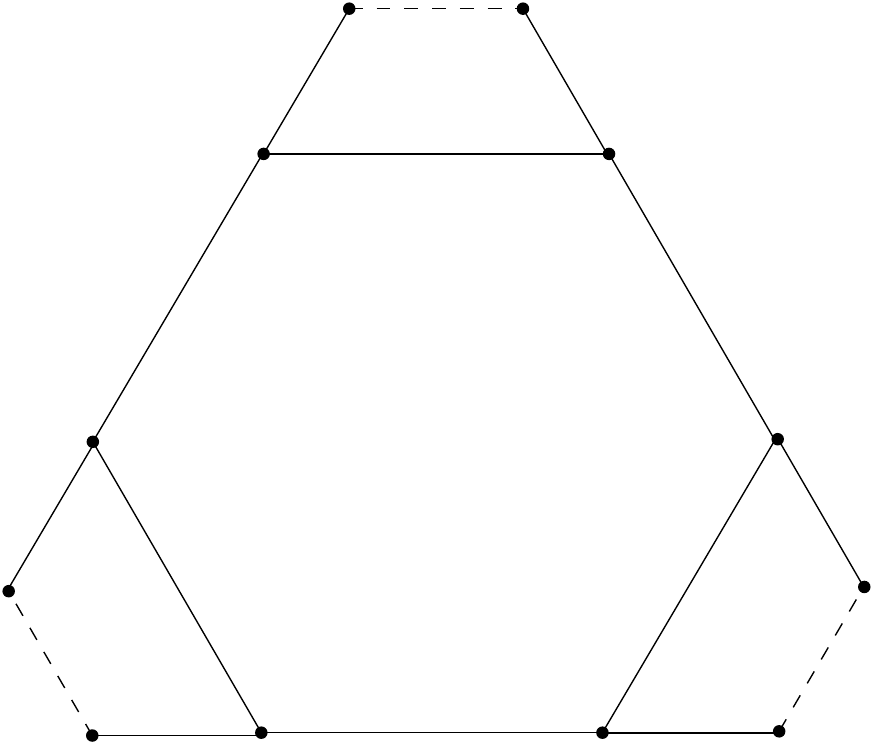}	
\end{center}
\caption{A typical torso of the Wild Cycle Graph}\label{fig:wildtorso}
\end{figure}

We shall see later that this is a particularly simple example of a tree structure of width 2, but it illustrates that the topological space $\Omega(G)$ may still be rich enough in such cases to support very complicated subsets $\Psi$. We end this section by showing that the construction outlined above does capture the $\Psi$-systems of graphs in the width 2 case.

\begin{lem}\label{treefromgraph2sums}
Let $G$ be a graph, and let $T$ be a tree structure on $G$ of width 2. Let $\Psi$ be a set of ends of $G$. Let $G'$ be the graph obtained from $G$ by subdividing each edge which has endpoints in different nodes of $T$.\footnote{formally, we add a new vertex $v_e$ corresponding to each such edge $e = vv'$, and replace $e$ in the set of edges by the two new edges $vv_e$ and $v'v_e$.}Then the $\Psi$-circuits of $G'$ are exactly the $\Psi$-circuits of $\Tcal(G, T)$ and the $\Psi^{\complement}$-bonds of $G'$ are exactly the $\Psi\ct$-cocircuits of $\Tcal(G, T)$.
\end{lem}
\begin{proof}
First we show that every $\Psi^{\complement}$-bond of $G'$ is a $\Psi\ct$-cocircuit of $\Tcal(G, T)$. Let $\underline b$ be such a $\Psi^{\complement}$-bond. Let $X$ be the set of vertices of $G'$ on one side of $\underline b$. Let $D$ be the set of vertices $t$ of $T$ such that $\tau(t)$ contains a vertex from $X$ and a vertex not from $X$. Since both $X$ and $V(G') \setminus X$ are connected, $D$ is an intersection of 2 connected subsets of the tree $T$, and so is also connected. $D$ doesn't include a ray to any end in $\Psi$, because $\underline b$ is a $\Psi^{\complement}$-bond.

For each $t \in D$, let $b(t)$ be the $\tau(t)$-cut of edges of $\tau(t)$ with one endpoint in $X$ and the other not in $X$. Both sides of $b(t)$ are connected, since both $X$ and $V(G') \setminus X$ are, so $b(t)$ is a circuit of $(M(\tau(t))^*$. For any $t'$ adjacent to $t$ in $T$, let the shared dummy vertices of $\tau(t)$ and $\tau(t')$ be $v_e$ and $v_f$. If $t' \not \in D$ then $v_e$ and $v_f$ are on the same side of $\underline{b}$, so $e(tt') \not \in b(t)$. If $t' \in D$, then since both $X$ and $V(G') \setminus X$ are connected exactly one of $v_e$ or $v_f$ is in $X$, so $e(tt') \in b(t)$. Thus we obtain that $\underline b = \underline{(D, b)}$ is a $\Psi\ct$-cocircuit of $\Tcal(G, T)$.

Next, we show that every $\Psi$-circuit of $G'$ is a $\Psi$-circuit of $\Tcal(G, T)$. Let $\underline o$ be such a $\Psi$-circuit, and let $C$ be the set of vertices $t$ of $T$ such that $\underline o$ meets $\tau(t)$. For any $t \not \in C$, there can only be one component of $T - t$ meeting $C$, since the unions of these components are separated by $t$ in $G \setminus t$. Thus $C$ is a subtree of $T$. Any end in the closure of $C$ is also in the closure of $\underline o$ and so must lie in $\Psi$.

For any $t \in C$, let $o(t)$ be the union of $\underline o \cap E(\tau(t))$ with the set of all edges $ee'$ of $\tau(t)$ where $e$ and $e'$ are the two edges of $G$ with endpoints in both $t$ and $t'$ for some $t'$ adjacent to $t$ in $C$. Then every vertex of $\tau(t)$ has degree 0 or 2 with respect to $o(t)$: this is immediate for vertices in $t$, and vertices given by edges with one endpoint in $t$ and the other in $t'$ have degree 0 if $t' \not \in C$, 2 if $t' \in C$. To show that $o(t)$ is a circuit, it remains to show that it is connected. Suppose not, for a contradiction. Then there is a cut $b$ of $\tau(t)$ not meeting $o(t)$ but with edges of $o(t)$ on both sides, so there is such a cut that doesn't contain any dummy edges. This cut is a finite cut of $G$ not meeting $\underline o$ but with edges of $\underline o$ on both sides, which is the desired contradiction. Thus each $o(t)$ is a circuit in $M(\tau(t))$. Thus we obtain that $\underline o = \underline{(C, o)}$ is a $\Psi$-circuit of $\Tcal(G, 
T)$.

To show that every $\Psi\ct$-cocircuit $\underline{(D, b)}$ of $\Tcal(G, T)$ is a $\Psi^{\complement}$-bond of $G'$, we pick any edge $e_0 \in \underline{(D, b)}$ and let $X$ and $Y$ be the sets of vertices in the same connected components of $G' \setminus \underline{(D, b)}$ as the endvertices $x_0$, $y_0$ of $e_0$. If $X = Y$ then there is a finite circuit in $G'$ meeting $\underline{(D, b)}$ just once, which is impossible by the argument above and \autoref{2SumsO1}. Let $t_0$ be the vertex of $T$ with $e_0 \in \tau(t_0)$. We prove by induction on the distance of $t$ from $t_0$ that $X \cup Y$ includes all vertices of $\tau(t)$ and if $t \in D$ then $b(t)$ is the set of edges of $\tau(t)$ with one end in $X$ and the other in $Y$. This is immediate if $t = t_0$, since $b(t_0)$ is a bond of $\tau(t_0)$. For any other $t' \in V(T)$, let $t$ be the neighbour of $t'$ in the direction of $t_0$. If $t' \in D$ then also $t \in D$ and so of the two dummy vertices shared by $\tau(t)$ and $\tau(t')$ one is in $X$ and 
the other in $Y$, giving the result since $b(t')$ is a bond of $\tau(t')$. If $t' \not \in D$ then the two dummy vertices shared by $\tau(t)$ and $\tau(t')$ are either both in $X$ or both in $Y$, so either all vertices of $\tau(t')$ are in $X$ or all of them are in $Y$. This shows that $\underline{(D, b)}$ is the bond of $G'$ consisting of all edges with one end in $X$ and the other in $Y$. It is a $\Psi^{\complement}$-bond since every end in its closure is in the closure of $D$ and so is in $\Psi^{\complement}$. 

Finally, we show that every $\Psi$-circuit of $\Tcal(G, T)$ is a $\Psi$-circuit of $G'$. Consider such a circuit $\underline{(C, o)}$. By the above argument and \autoref{2SumsO1} it never meets a finite bond of $G'$ just once and so, by \autoref{is_scrawl} applied to the topological cycle matroid of $G'$ it is a union of topological circuits. To show that it is the edge set of a single topological circle, it is enough by \autoref{is_circuit} to show that for any $e, f \in \underline{(C, o)}$ there is a finite bond $b$ of $G$ with $b \cap \underline{(C, o)} = \{e, f\}$. Consider the unique finite path $t_1, ... t_n$ in $T$ with $e \in E(\tau(t_1))$ and $f \in E(\tau(t_n))$. Let $e_0 = e, e_n = f$ and for $0 < i < n$ let $e_i = e(t_it_{i+1})$. For each $i \leq n$ we let $b_i$ be any bond of $\tau(t_i)$ with $b_i \cap o(t_i) = \{e_{i-1}, e_i\}$. Without loss of generality we may choose the $b_i$ to contain no dummy edges other than the $e_i$. Then $\bigcup_{i = 1}^n b_i \setminus E$ is the desired finite bond 
of $G$. Thus $\underline{(C, o)}$ is a topological circuit of $G$. It is a $\Psi$-circuit since every end in its closure is in the closure of $C$ and so is in $\Psi$.
\end{proof}

\section{Determinacy and $(O2)$ for trees of matroids of overlap 1}\label{sec:2sums}

In \autoref{orth}, we saw that $(O2)$ corresponds, for $\Psi$-systems, to a principle implying path-connectedness from connectedness. Here we will show that, for the systems arising from trees of matroids, $(O2)$ has close links with determinacy of games. We begin by analysing an illuminating example.

Let $\Tcal^{game}$ be the tree of matroids given by $(T_2, M^{game})$, as follows: $T_2$ is the infinite rooted binary tree (to fix notation, we take the vertices of $T_2$ to be the finite sequences from $\{0, 1\}$, with $s$ adjacent to each of $s0$ and $s1$ for any such sequence $s$, and we call the empty sequence $\emptyset$). For any node $s$ of $T_2$, we take the ground set of $M^{game}(s)$ to be $\{d_s, d_{s0}, d_{s1}\}$ and we take $M^{game}(s)$ to be uniform, of rank 1 if the length of $s$ is even and of rank 2 if the length of $s$ is odd. This tree of matroids has overlap 1, with all edges except $d_{\emptyset}$ being dummy edges. The ground set $E^{game}$ of $\Tcal^{game}$ is simply $\{d_{\emptyset}\}$. The structure of this tree of matroids is displayed in \autoref{fig:treeofmats}.

\begin{figure}
\begin{center}
 \includegraphics[height=6cm]{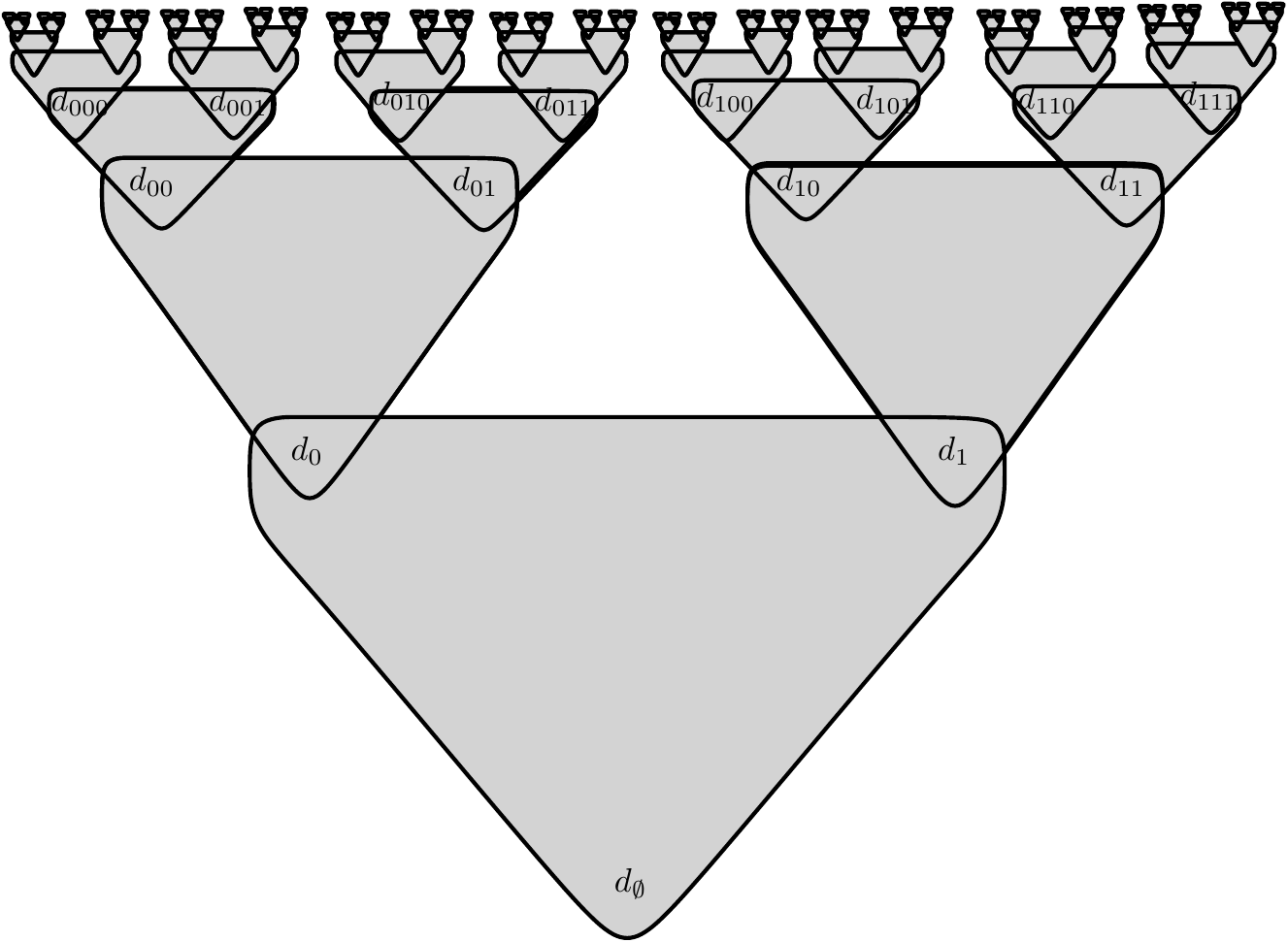}
\end{center} 
\caption{The tree of matroids $\Tcal^{game}$}	\label{fig:treeofmats}
\end{figure}

Although the ground set has only 1 element, so that the sets of $\Psi$-circuits of $\Tcal$ or $\Tcal^*$ must be very simple for any $\Psi$, our analysis of $(O2)$ will still be complex because of the way in which these sets arise from $\Tcal$. Any instance of $(O2)$ for trees of matroids is reducible to one on which the ground set has only one element, since $(O2)$ holds for the partition $E = \{e\} \dot \cup P_{Co} \dot \cup P_{De}$ of the ground set of $\Tcal$ if and only if it holds for the partition $\{e\} = \{e\} \dot \cup \varnothing \dot \cup \varnothing$ of the ground set of $\Tcal/P_{Co}\backslash P_{De}$. However, as this section will illustrate, this reduction does not diminish the complexity of the problem.

Let's fix some set $\Psi \se \{0, 1\}^{\Nbb}$ and examine the meaning of $(O2)$ applied to the $\Psi$-circuits and $\Psi\ct$-cocircuits of $\Tcal^{game}$, with the partition $E^{game} = \{d_{\emptyset}\} \dot \cup \varnothing \dot \cup \varnothing$. If $(O2)$ is true, then one of the following 2 things happens:
\begin{enumerate}
\item There is a $\Psi$-circuit through $d_{\emptyset}$.
\item There is a $\Psi\ct$-cocircuit through $d_{\emptyset}$.
\end{enumerate}

Let's think first of all about $(1)$. This says that we can find a $\Psi$-precircuit $(C, o)$ with $\emptyset \in C$, $d_{\emptyset} \in o(\emptyset)$.
The shape of $C$ is now quite constrained. For any $s \in C$ we have $d_s \in o(s)$. If $s$ has even length, then $o(s)$ can only be $\{d_s, d_{s0}\}$ or $\{d_s, d_{s1}\}$. On the other hand, if $s$ has odd length then $o(s)$ can only be $\{d_s, d_{s0}, d_{s1}\}$. Thus $C$ is a set of finite sequences from $\{0, 1\}$ with the following properties:
\begin{itemize}
\itum $\emptyset \in C$.
\itum $C$ is closed under taking initial segments.
\itum For any $s \in C$ of even length, exactly one of $s0$ and $s1$ is in $C$.
\itum For any $s \in C$ of odd length, both of $s0$ and $s1$ are in $C$.
\itum For any $s \in \{0, 1\}^{\Nbb}$ such that all finite initial segments of $s$ are in $C$, $s \in \Psi$.
\end{itemize}  
These properties collectively state that $C$ gives a winning strategy for the first player in the game $\Gcal(\Psi)$ from the introduction, with $\Psi$ considered as a subset of $\{0,1\}^\Nbb$: the first player should play so as to ensure that the finite sequence generated so far always remains in $s$. Conversely, given a set $C$ with these properties, we can define a function $o$ on $C$ sending $s$ to $\{d_s, d_{s0}\}$ if $s$ has even length and $s0 \in C$, to $\{d_s, d_{s1}\}$ if $s$ has even length and $s1 \in C$, and to $\{d_s, d_{s0}, s_{s1}\}$ if $s$ has odd length. Then $\underline{(C, o)}$ is a $\Psi$-circuit of $\Tcal^{game}$ with $\underline{(C, o)}$ witnessing $(1)$.

What this shows is that $(1)$ is equivalent to the statement that the first player has a winning strategy in the game $\Gcal(\Psi)$. A similar argument shows that $(2)$ is equivalent to the statement that the second player has a winning strategy in that game. Thus in this case $(O2)$ is equivalent to determinacy of the game $\Gcal(\Psi)$. By introducing some slightly more complex games, we will now show that for any tree $\Tcal$ of matroids of overlap 1 and any set $\Psi$ of ends of $\Tcal$ there is a collection of games such that $(\Tcal, \Psi)$ induces a matroid if and only if all of the games in that collection are determined.

We temporarily fix such a $\Tcal$ and $\Psi$, together with a partition $E = \{e\} \dot \cup P_{Co} \dot \cup P_{De}$ of the ground set of $\Tcal$. Let $t_0$ be the node of $T$ such that $e \in E(t_0)$.

\begin{dfn}
The {\em circuit game} $\Gcal = \Gcal(\Tcal, \Psi, P_{Co}, P_{De})$ is played between two players, called Sarah and Colin\footnote{The name `Sarah' has been chosen because it sounds similar to `circuit', and `Colin' because it may be pronounced co-lin, to sound a bit like `cocircuit'}, as follows:

Play alternates between the players, with Sarah making the first move. At any point in the game there is a {\em current node} $t_c \in V(t)$, and a {\em current edge} $e_c \in E(t_c)$. Initially we set $t_c = t_0$ and $e_c=e$ to be the node of $T$ with $e_c \in E(t_c)$. For any $n$ the $(2n-1)$\textsuperscript{st} move is made by Sarah: she must play a circuit $o_n$ of $M(t_c)$ such that $e_c \in o_n$ but $o_n \cap P_{De} = \varnothing$. Then the $2n$\textsuperscript{th} move is made by Colin: he must play a node $t_n$ adjacent to $t_c$ and further from $t_0$ than $t_c$ is, such that $e(t_ct_n) \in o_n$. After he does this, the current node is updated to $t_n$, and the current edge to $e(t_{n-1}t_n)$. If play continues forever, then Sarah wins if the end $\omega$ of $T$ containing $(t_n | n \in \Nbb)$ is in $\Psi$, and Colin wins if $\omega \in \Psi^{\complement}$.

The {\em cocircuit game} $\Gcal^* = \Gcal^*(\Tcal, \Psi, P_{Co}, P_{De})$ is the game like the dual circuit game $\Gcal(\Tcal^*, \Psi\ct, P_{De}, P_{Co})$, but with the roles of Sarah and Colin reversed. We will also use a different notation for the cocircuit game, putting stars on the notation for the circuit game. Thus for example the current edge is denoted $e_c^*$ and Colin's $n$\textsuperscript{th} move is denoted $o_n^*$.
\end{dfn}

\begin{lem}
Sarah has a winning strategy in $\Gcal$ if and only if there is a $\Psi$-circuit $\underline{(C, o)}$ of $\Tcal$ with $e \in \underline{(C, o)} \subseteq \{e\} \dot\cup P_{Co}$.
\end{lem}
\begin{proof}
Suppose first that there is such a $\Psi$-circuit $\underline{(C, o)}$. Then Sarah can win in $\Gcal$ by always choosing $o(t_c)$ when it is her turn to play.

Suppose for the converse that Sarah has a winning strategy $\sigma$ in $\Gcal$. Let $C$ be the set of nodes $t$ of $T$ such that there is some finite play according to $\sigma$ consisting of $2n + 1$ moves for some $n$ after which $t$ is the current node. Then this play is unique, since Sarah's moves are determined by $\sigma$, and Colin's must be the sequence of vertices along the finite path in $T$ from $t_0$ to $t$. We set $o(t)$ to be the final move $o_n$ made by Sarah in that play. It is immediate that $(C, o)$ is a $\Psi$-pre-circuit of $\Tcal$ with the desired properties.
\end{proof}

\begin{cor}
Colin has a winning strategy in $\Gcal^*$ if and only if there is a $\Psi\ct$-cocircuit $\underline{(C, o)}$ of $\Tcal$ with $e \in \underline{(C, o)} \subseteq \{e\} \dot\cup P_{De}$.\qed
\end{cor}

In order to relate $(O2)$ to determinacy of $\Gcal$, we need to show that $\Gcal$ and $\Gcal^*$ are closely related games.

\begin{lem}\label{2sumtransfer}
Colin has a winning strategy in $\Gcal$ if and only if he has one in $\Gcal^*$.
\end{lem}
\begin{proof}
For the `if' part, suppose that he has a winning strategy $\sigma^*$ in $\Gcal^*$. Then he can win in $\Gcal$ by playing as follows:

He should imagine an auxilliary play in the game $\Gcal^*$, in which he plays according to $\sigma^*$, and for which he should ensure that at any point the current edge and node agree with those in $\Gcal$. When Sarah makes the move $o_n$, he should pick some edge in $o_n \cap o_n^*$ other than $e_n$ (there is such an edge by \autoref{is_scrawl}). This edge $t$ can then only be a dummy edge $e(t_ct)$ for some $t$ adjacent to $t_c$. He should play $t$ as $t_n$ in $\Gcal$ and imagine that Sarah also plays $t$ as $t_n^*$ in $\Gcal^*$. If play continues forever, then the end $\omega$ containing $(t_n | n \in \Nbb)$ is in $\Psi^{\complement}$ since $\sigma^*$ is winning.

For the `only if' part, suppose that he has a winning strategy $\sigma$ in $\Gcal$. Then he can win in $\Gcal^*$ by playing as follows:

He should imagine an auxilliary play in the game $\Gcal$, in which he plays according to $\sigma$, and for which he should ensure that at any point the current edge and node agree with those in $\Gcal^*$. When he has to make a move $o_n^*$, he should consider the set $R$ of responses prescribed by $\sigma$ to legal moves $o_n$ that Sarah could make in $\Gcal$. Then $R \cup P_{De}$ meets every circuit $o$ of $M(t_c)$ with $e_c \in o$. Thus since $(O2)$ holds for the matroid $M(t_c)$ there is some cocircuit $o_n^*$ of that matroid with $e_c \in o_n^* \subseteq \{e_c\} \cup R \cup P_{De}$, and Colin should play such a cocircuit. If Sarah responds by playing $t_n^*$, then we must have $t_n^* \in R$ and so there is some legal move $o_n$ in $\Gcal$ to which $\sigma$ prescribes the response $t_n^*$. Then Colin should imagine that, in the play of $\Gcal$, Sarah plays $o_n$ and he responds by playing $t_n^*$ as $t_n$. If play continues forever, then the end $\omega$ containing $(t_n^* | n \in \Nbb)$ is in $\
\Psi^{\complement}$ since $\sigma$ is winning.
\end{proof}

\begin{cor}\label{2sumO2}
$(O2)$ holds for the partition $E = \{e\} \dot\cup P_{Co} \dot\cup P_{De}$ of the groundset of $\Tcal$ if and only if $\Gcal(\Tcal, \Psi, P_{Co}, P_{De})$ is determined.\qed
\end{cor}

Since any game $\Gcal(\Psi)$ with  $\Psi\se A^\Nbb$ and $A$ countable can be coded by such a game with $A=\{0,1\}$, we also get:

\begin{cor}
The Axiom of Determinacy is equivalent to the statement that every set $\Psi$ of ends of every tree of finite matroids of overlap 1 induces a matroid. If the Axiom of Choice holds then there is a tree of finite matroids of overlap 1 and a set $\Psi$ of ends of that tree that doesn't induce a matroid. \qed
\end{cor}

\begin{cor}\label{Borel2sums}
For any tree of countable tame matroids $\Tcal = (T, M)$ of overlap 1 and any Borel set $\Psi$ of ends of $T$, the pair $(\Tcal, \Psi)$ induces a matroid.
\end{cor}
\begin{proof}
This is immediate from Borel determinacy, \autoref{2sumO2} and the fact that for each partition of the ground set as $\{e\} \dot \cup P_{Co} \dot \cup P_{De}$ the projection map from the set of legal infinite plays in $\Gcal(\Tcal, \Psi, P_{Co}, P_{De})$ to $\Omega(T)$ sending a play to the end containing the sequence $(t_n | n \in \Nbb)$ for that play is continuous.
\end{proof}

In \autoref{sec:arbi_sums} we will extend these techniques to trees of finite representable matroids and so get results applying to all locally finite graphs. However, our results so far already have implications for graphs with a tree structure of width 2.

\begin{thm}\label{Borelwidth2}
Let $G$ be a graph with a tree structure $T$ of width 2, and $\Psi$ a Borel set of ends of $G$. Then $(G, \Psi)$ induces a matroid.
\end{thm}
\begin{proof}
Let $G'$ be obtained from $G$ by subdividing certain edges as in the proof of \autoref{treefromgraph2sums}. Then by \autoref{Borel2sums}, $(\Tcal(G, T), \Psi)$ induces a matroid $M$, which by \autoref{treefromgraph2sums} is also induced by $(G', \Psi)$. Then the matroid obtained from $M$ by contracting one of each pair of edges subdividing an edge of $G$ is induced by $(G, \Psi)$.
\end{proof}

Assuming the Axiom of Choice holds, we can also give another example of a graph $G$ and a set of ends $\Psi$ of $G$ such that $(G, \Psi)$ doesn't induce a matroid.

\begin{eg}\label{eg:bad_torsos}
\autoref{3col} illustrates that we may 3-colour the edges of $T_2$ in such a way that the edges incident with any vertex $s$ are the same colour if $s$ has even length considered as a finite $\{0,1\}$-sequence, but are all different colours if $s$ has odd length. 

\begin{figure}
\begin{center}
 \includegraphics[height=4cm]{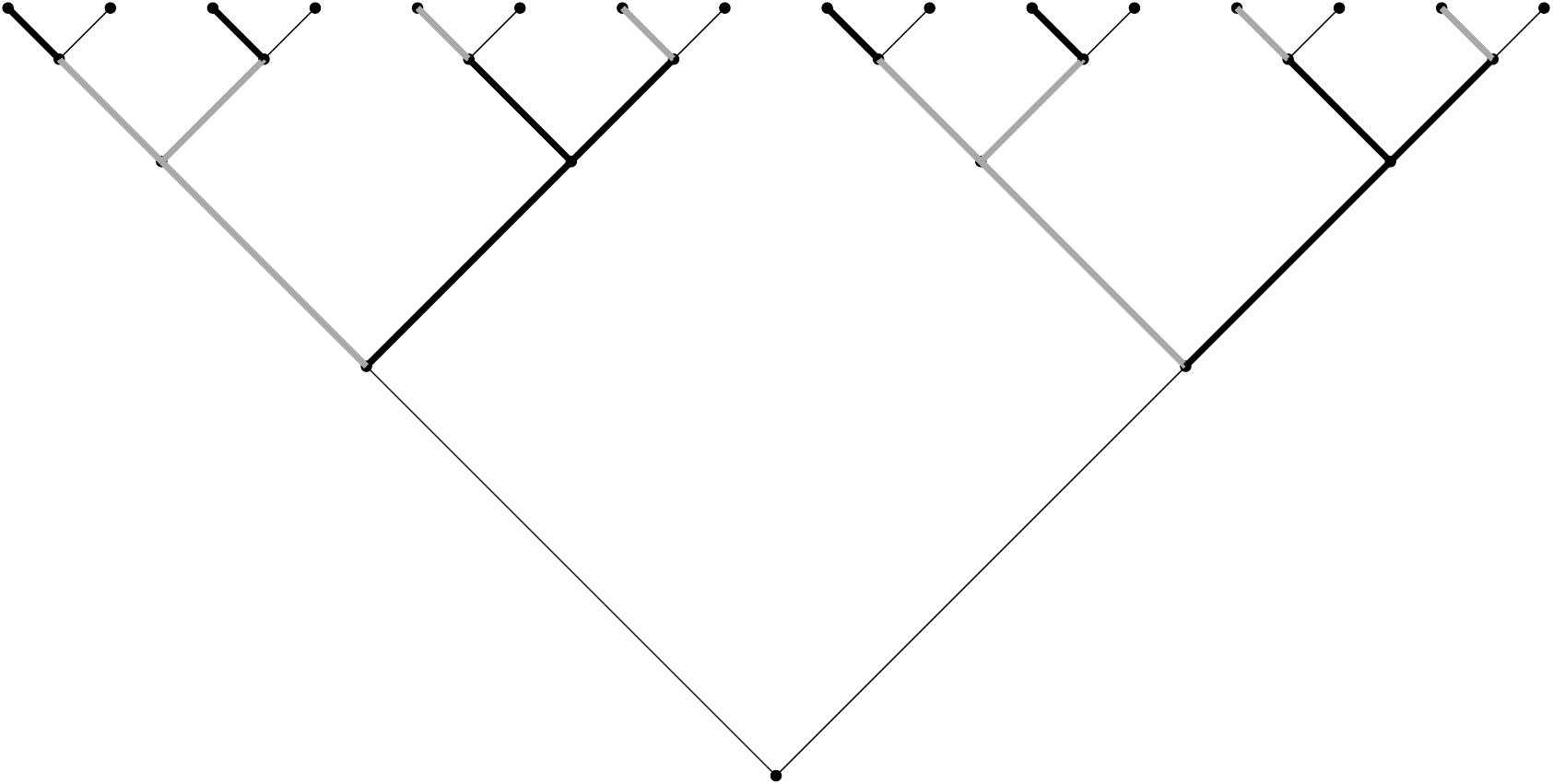}
\end{center}
\caption{The binary tree with a particular $3$-coloring of its edges.}\label{3col}
\end{figure}
We fix such a 3-colouring given as a function $c \colon E(T_2) \to V(K_3)$. Let $G$ be the graph obtained from $T_2 \times K_3$ by removing all edges of the form $e \times \{c(e)\}$ with $e \in E(T_2)$. Then $G$ has a tree structure of width 2, in which the vertices of $T$ are the sets $\{s\} \times V(K_3)$ with $s$ a vertex of $T_2$. The shapes of the torsos for this tree structure are given in \autoref{fig:badtorsos}. 

\begin{figure}
\begin{center}
 \includegraphics[height=4cm]{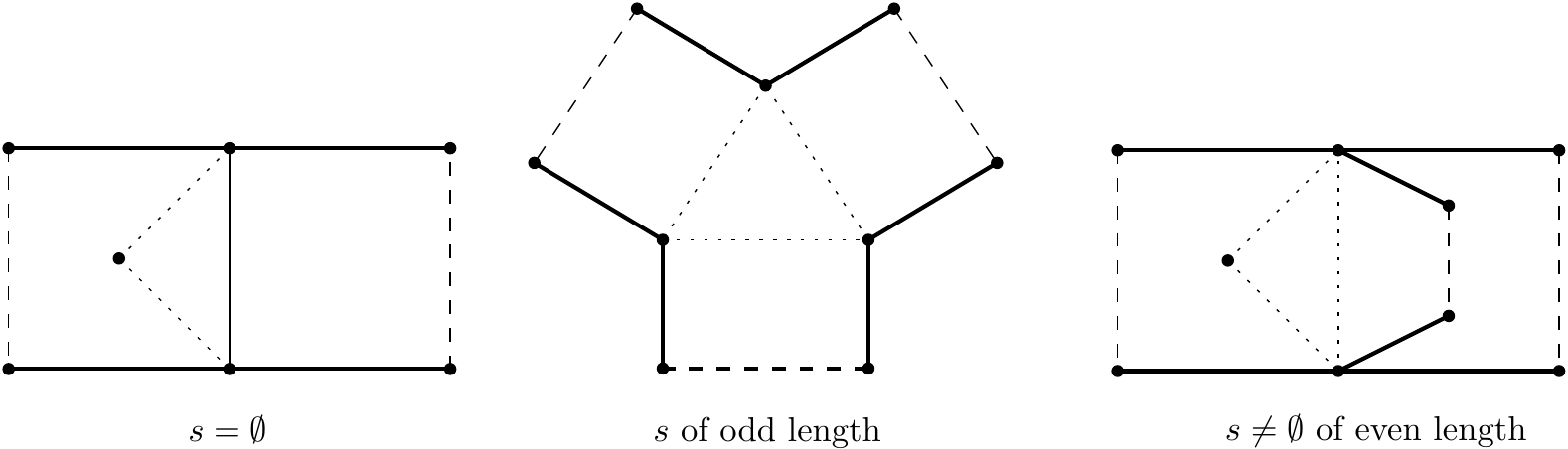}
\end{center}
\caption{The Torsos from \autoref{eg:bad_torsos}.}\label{fig:badtorsos}
\end{figure}

Let $\Psi$ be a set of ends of $G$ such that $\Gcal(\Psi)$ is not determined. Then the tree of matroids obtained from $\Tcal(G, T)$ by contracting the bold edges in \autoref{fig:badtorsos} and deleting the dotted edges is isomorphic to $\Tcal^{game}$, and we know $(\Tcal^{game}, \Psi)$ does not induce a matroid. Thus $(\Tcal(G, T), \Psi)$ does not induce a matroid, and so $(G, \Psi)$ cannot induce a matroid, and so $(T_2 \times K_3, \Psi)$ does not induce a matroid.
\end{eg}

Now we can explain the sense in which we said that the wild cycle graph was relatively simple when we discussed it in \autoref{treesofmatroids}.

\begin{lem}
For any set $\Psi$ of ends of the wild cycle graph $G_{wild}$, the pair $(G_{wild}, \Psi)$ induces a matroid. 
\end{lem}
\begin{proof}
As in the proof of \autoref{Borelwidth2}, it is enough to check that $(\Tcal(G_{wild}, T), \Psi)$ induces a matroid, where $T$ is the tree structure from \autoref{eg:wildstruc}. Now we may note that the torsos for this tree structure, depicted in \autoref{fig:wildtorso}, have the property that no bond contains more than 2 dummy edges. Thus in the cocircuit games for this tree of matroids, all of Sarah's moves apart from her first one are forced. Thus all these games are determined, and we are done by \autoref{2sumtransfer} and \autoref{2sumO2}.
\end{proof}

There are other simple examples of graphs which induce a matroid for any $\Psi$:

\begin{lem}\label{T_2_lem}
Let $T$ be any locally finite tree, and let $\Psi$ be any set of ends of $T \times K_2$. Then $(T \times K_2, \Psi)$ induces a matroid.
\end{lem}
\begin{proof}
Once more it is enough to check that $(\Tcal(T \times K_2, T'), \Psi)$ induces a matroid, where $T'$ is the tree structure whose vertices are the sets $\{t\} \times V(K_2)$ for $t \in V(T)$. The torsos are of the form $S \times K_2$, where $S$ is a finite star. They have the property that no circuit contains more than 2 dummy edges, and so all the circuit games for this tree of matroids are determined, and we are done by \autoref{2sumO2}.
\end{proof}

\section{Trees of matroids II}\label{tm2}

To capture graphs which cannot be given a tree structure of width 2, we need a more general notion of pasting in a tree of matroids, for which we will work with representable matroids. Strictly speaking, we will be pasting together {\em represented} matroids, since the matroid structure after pasting can depend on the choices of representation before pasting. 

Before returning to trees of matroids, we shall first outline how to paste together just 2 matroids in this way. We shall take a slightly unusual point of view on representations: we think of a representation of a finite matroid $M$ over a field $k$ as given by a subspace $U$ of $k^{E(M)}$ such that the minimal nonempty supports of elements in $U$ are the $M$-circuits (there is such a subspace if and only if $M$ is representable in the usual sense over $k$). The dual of $M$ is then represented by the orthogonal complement $U^{\perp}$ of $U$.

Now suppose that we have two finite matroids $M_1$ and $M_2$ where $M_i$ has ground set $E_i$ and is represented over $k$ by a subspace $U_i$ of $k^{E_i}$. Then there are canonical embeddings of $U_1$, $U_2$ and $k^{E_1 \triangle E_2}$ as subspaces of $V = k^{E_1 \cup E_2}$. We let $U_1 \triangle U_2$ be $(U_1 + U_2) \cap k^{E_1 \triangle E_2}$: the vectors in this space are those $v$ such that there are $v_1 \in U_1$ and $v_2 \in U_2$ with $v_1 \restric_{E_1 \cap E_2} = -v_2 \restric_{E1 \cap E_2}$, $v \restric_{E_1 \sm E_2} = v_1 \restric_{E_1 \sm E_2}$ and $v \restric_{E_2 \sm E_1} = v_2 \restric_{E_2 \sm E_1}$. 

This construction is well behaved with respect to duality. The orthogonal complement of $U_1 \triangle U_2$ in $V$ is $(U_1^{\perp} \cap U_2^{\perp}) + \left(k^{E_1 \triangle E_2}\right)^{\perp} = (U_1^{\perp} \cap U_2^{\perp}) + k^{E_1 \cap E_2}$. So the orthogonal complement of $U_1 \triangle U_2$ in $k^{E_1 \triangle E_2}$ is the intersection of that space with $k^{E_1 \triangle E_2}$, which is the set of those $w$ such there are $w_1 \in U_1^{\perp}$ and $w_2 \in U_2^{\perp}$ with $w_1 \restric_{E_1 \cap E_2} = w_2 \restric_{E_1 \cap E_2}$, $w \restric_{E_1 \sm E_2} = w_1 \restric_{E_1 \sm E_2}$ and $w \restric_{E_2 \sm E_1} = w_2 \restric_{E_2 \sm E_1}$. This isn't quite the same as $U_1^{\perp} \triangle U_2^{\perp}$ - there is a missing minus sign in one of the equations - but the supports of the vectors, and so the induced matroids, are the same. Thus we have $(M_1 \triangle M_2)^* = M_1^* \triangle M_2^*$.

This construction also allows us to glue together pairs of tame thin sums matroids, provided that the overlap of their ground sets is finite. The details are beyond the scope of this paper, but the basic reason is that in proving (O2), which is potentially the trickiest of the axioms, it is possible by contracting $P$ and deleting $Q$ to reduce the problem to one on the finite set consisting of $e$ and the edges in the overlap set.

If we want to use a construction like this to glue together a tree of matroids, we will need a representation of each of the (finite) matroids.

\begin{dfn}
Let $k$ be a finite field.
A {\em $k$-representation} of a tree $(T, M)$ of matroids is a function $V$ assigning to each vertex $t$ of $T$ a subspace $V(t)$ of $k^{E(t)}$ such that $M(V(t)) = M(t)$.
The {\em dual} $V^{\perp}$ of such a $k$-representation is the representation of the dual tree of matroids which assigns to each node $t$ of $T$ the space $V(t)^{\perp}$. We will only ever consider representations of trees of finite matroids.

In this context, a {\em $\Psi$-vector of $V$} consists of a function $v$ assigning to each vertex $t$ of $T$ a vector $v(t) \in V(t)$, in such a way that for any edge $tt'$ of $T$ we have $v(t) \restric_{E(tt')} = v(t') \restric_{E(tt')}$ and that every end of $T$ in the closure of $\{t \in V(T) | v(t) \neq 0\}$ is in $\Psi$. The set of such $\Psi$-vectors is denoted $\Vcal(V, \Psi)$. The {\em support} of a $\Psi$-vector $v$ is the set $\underline v = E \cap \bigcup_{t \in T} \underline{v(t)}$. The set of such supports is denoted $\underline{\Vcal}(V, \Psi)$.

We say that $(V, \Psi)$ {\em induces a matroid} $M = M(V)$ if $\Ccal(M) \subseteq \underline{\Vcal}(V, \Psi) \subseteq \Scal(M)$ and $\Ccal(M^*) \subseteq \underline{\Vcal}(V^\perp, \Psi\ct) \subseteq \Scal(M^*)$.
\end{dfn}

The question of when $(O2)$ holds for these systems is once more tricky, and will be addressed in \autoref{sec:arbi_sums}. However, we are already in a position to give a simple proof of $(O1)$ and of tameness.

\begin{lem}[$(O1)$ and tameness for representable trees of finite matroids]\label{repO1tame}
Let $\Tcal = (T, M)$ be a tree of finite matroids with a $k$-representation $V$, let $\Psi$ be a set of ends of $T$, and let $v$ and $w$ be respectively a $\Psi$-vectors of $V$ and a $\Psi\ct$-vector of $V^{\perp}$. Then $|\underline v \cap \underline w|$ is finite but not equal to 1.
\end{lem}
\begin{proof}
If it were infinite, then there would be an end $\omega$ in the closure of $\underline v \cap \underline w$ and so in the closure of both $\{t \in V(T) | v(t) \neq 0\}$ and $\{t \in V(T) | w(t) \neq 0\}$, so that $\omega$ would have to be in both $\Psi$ and $\Psi^{\complement}$, a contradiction. So it is finite.

Now fix some node $t_0$ of $T$ and for any node $t$ let $d(t)$ be the distance from $t_0$ to $t$ in $T$ (thus $d(t_0) = 0$). Let $\hat v \colon E \to k$ be the function sending $e \in E(t)$ to $v(t)(e)$, and $\hat w \colon E \to k$ be the function sending $e \in E(t)$ to $(-1)^{d(t)} w(t)(e)$. Then we have
\begin{eqnarray*}
\sum_{e \in E} \hat v(t) \hat w (t) &=& \sum_{t \in V(T)}(-1)^{d(t)}\left(\sum_{e \in E(t)} v(t)(e) w(t)(e)  - \sum_{tt' \in E(T)}\sum_{e \in E(tt')} v(t)(e)w(t)(e)\right) \\
&=& - \sum_{tt' \in E(T)} \left((-1)^{d(t)} + (-1)^{d(t')}\right) \left(\sum_{e \in E(tt')} v(t)(e) w(t)(e)\right) \\
&=& 0 \\
\end{eqnarray*}
and it follows that $|\underline v \cap \underline w| = |\underline{\hat v} \cap \underline{\hat w}| \neq 1$.
\end{proof}

\begin{rem}\label{tree_thin_sum}
Once we have shown that this system induces a (tame) matroid $M$, the proof above will also show that it is a thin sums matroid over $k$ according to the characterisation given in 
\cite{THINSUMS}, since we can choose the function $c_{\underline v} \colon \underline v \to k$ for a circuit $\underline v$ to be given by $\hat v \restric_{\underline v}$ and similarly take $d_{\underline w} = \hat w \restric_{\underline w}$.
\end{rem}

With this new construction, we can capture the $\Psi$-system of any graph with a tree decomposition.

\begin{dfn}
For a graph $G$ with a tree structure $T$, let $V(G, T)$ be the unique representation of $\Tcal(G, T)$ over $\Fbb_2$ (such a representation exists since for each $t \in V(t)$ the matroid $M(\tau(t))$ is graphic and so binary).
\end{dfn}

\begin{lem}\label{getcircs:rep}
Let $G$ be a graph, and let $T$ be a tree structure on $G$. Let $\Psi$ be a set of ends of $G$. Let $G'$ be the graph obtained from $G$ by subdividing each edge which has endpoints in different nodes of $T$.\footnote{as before, we add a new vertex $v_e$ corresponding to each such edge $e = vv'$, and replace $e$ in the set of edges by the two new edges $vv_e$ and $v'v_e$.}Then every $\Psi$-circuit of $G'$ is the support of a $\Psi$-vector of $V(G, T)$ and every $\Psi^{\complement}$-bond of $G'$ is the support of a $\Psi\ct$-vector of $(V(G, T))^{\perp}$.
\end{lem}
\begin{proof}
First we show that every $\Psi^{\complement}$-bond of $G'$ is the support of a vector of $(V(G, T, \Psi))^*$. Let $\underline b$ be such a $\Psi^{\complement}$-bond. Let $X$ be the set of vertices of $G'$ on one side of $\underline b$. For each $t \in V(T)$, let $b(t)$ be the $\tau(t)$-cut of edges of $\tau(t)$ with one endpoint in $X$ and the other not in $X$, and let $w(t)$ be the characteristic function of $b(t)$: thus $w$ is a vector of $(V(G, T, \Psi))^{\perp}$. Then $\underline b = \underline w$.

Next, we show that every $\Psi$-circuit of $G'$ is the support of a vector of $V(G, T, \Psi)$. Let $\underline o$ be such a $\Psi$-circuit, and let $O$ be the circle in $\tilde G_{\Psi}$ inducing $\underline o$. Fix some vertex $t_0$ of $T$ such that $\underline o$ meets $E(t_0)$. For any other vertex $t$ of $T$ let $T\upset t$ be the set of vertices $t'$ of $t$ on the other side of $t$ from $t_0$, together with $t$ itself. Let $E \upset t$ be $E \cap \bigcup_{t' \in T \upset t} E(t')$. For any $tt' \in E(T)$, with $t'$ further from $t_0$ than $t$, let $F(tt') \subseteq E(tt')$ be the set of those edges $v_ev_f$ such that there is an arc in $O$ from $v_e$ to $v_f$ using only edges of $E \upset t'$. For any vertex $t$ of $T$, let $o(t)$ be the union of $\underline o \cap E(t)$ with all of the $F(tt')$ for $t'$ adjacent to $t$ in $T$, and let $v(t)$ be the characteristic function of $o(t)$. Since $\underline o = \underline v$, it suffices to prove that $\underline v$ is a $\Psi$-vector. Every end in 
the closure of $\{t \in V(t) | v(t) \neq 0\}$ is in the closure of $o$ and so is in $\Psi$. So we just need to show that for each node $t$ of $T$ the function $v(t)$ is in the circuit space of $\tau(t)$. In fact, we shall show something stronger: that $o(t)$ is a vertex-disjoint union of circuits of $\tau(t)$.

The circle $O$ can be broken into finitely many arcs each of which uses either only edges in $E(t_0)$ or else only edges not in $E(t_0)$, with consecutive arcs around $O$ being of opposite types. For each arc using only edges not in $E(t_0)$ there is some $t'$ adjacent to $t_0$ such that that arc only uses edges from $E \upset t'$. Replacing each such arc with the corresponding edge in $F(t_0t')$ gives the set $o(t_0)$, which is therefore a circuit of $\tau(t_0)$. 

For any $t \neq t_0$, let $t_-$ be the neighbour of $t$ in the direction of $t_0$, and let $v_ev_f$ be any edge in $F(t_-t)$. Then there is an arc $A$ in $O$ from $v_e$ to $v_f$ using only edges of $E \upset t$. $A$ can be broken into finitely many arcs each of which uses either only edges in $E(t)$ or else only edges not in $E(t)$, with consecutive arcs along $A$ being of opposite types. For each arc using only edges not in $E(t)$ there is some $t'$ adjacent to $t$ such that that arc only uses edges from $E \upset t'$. Replacing each such arc with the corresponding edge in $F(tt')$ gives a path $P(ef)$ of $\tau(t)$, which together with $v_ev_f$ itself gives a circuit $o(v_ev_f)$ of $\tau(t)$. Then $o(t)$ is the union of the vertex-disjoint circuits $o(v_ev_f)$, completing the proof. 
\end{proof}

\begin{lem}
In the context of \autoref{getcircs:rep}, for any $\Psi$-vector $v$ of $V(G, T)$, $\underline v$ is a union of $\Psi$-circuits. For any vector $w$ of $(V(G, T, \Psi))^{\perp}$, $\underline w$ is a union of $\Psi^{\complement}$-bonds.
\end{lem}
\begin{proof}
By \autoref{getcircs:rep} and \autoref{repO1tame}, $\underline v$ never meets a finite bond of $G'$ just once and so, by \autoref{is_scrawl} it is a union of topological circuits of $G'$. Each such circuit is a $\Psi$-circuit since every end in its closure is in the closure of $\{t \in V(T) | v(t) \neq 0\}$ and so is in $\Psi$. The proof for $w$ is analogous.
\end{proof}

In fact, the results above apply to all locally finite graphs.

\begin{lem}
Any connected locally finite graph $G$ can be given a tree structure.
\end{lem}
\begin{proof}
Let $U$ be a normal spanning tree of $G$, with root node $v_0$. For any down-closed set $X$ of vertices of $G$ we take $\delta(X)$ to be the set of minimal vertices not in $X$ (here minimality is with respect to the tree order $\leq$ on $U$). For any set $X$ of vertices of $G$, let $X \downset$ be the down-closure of $X$ in $U$, and $N(X)$ the set of vertices adjacent to or in $X$. We build a sequence of finite subsets $V_n$ of the vertices of $G$ by setting $V_0 = \varnothing$ and $V_{n+1} = N(V_n) \downset \cup \delta(V_n)$. For any $n$ and any vertex $v \in \delta(V_n)$, we set $t(v) = \{v' \in V_{n+1} | v \leq v'\}$. Let $T$ be the set of sets $t(v)$ arising in this way. By construction, $T$ is a partition of the vertices of $T$ into finite, connected sets. We order the vertices of $T$ by $t(v) \leq t(v')$ if and only if $v \leq v'$ in the tree order on $N$. This gives a tree-order (with root $t(v_0)$) on $T$, making $T$ a tree. It remains to show that distinct vertices of $T$ are adjacent if and only if 
they 
contain 
adjacent vertices of $G$.

If $t(v)$ and $t(v')$ are adjacent in $T$, with $v < v'$, then let $n$ be such that $v \in \delta(V_n)$. As $v < v'$, $v' \not \in V_n \cup \delta(V_n)$ so $v' \not \in V_{n+1}$. Let $w$ be minimal such that $v < w \leq v'$ and $w \not \in V_{n+1}$. Then $w \in \delta(V_{n+1})$ and we have $t(v) < t(w) \leq t(v')$ in $T$, so $w = v'$. Thus the predecessor $v^-$ of $v'$ in $U$ is in $V_{n+1}$, but it can't be in $V_{n}$ since $v' > v$. So $v^- \in t(v)$ and so there is an edge from $t(v)$ to $t(v')$.

Now let $v \neq v'$ be such that there is an edge from $t(v)$ to $t(v')$ in $G$. Say the endpoints of this edge are $w \in t(v)$ and $w' \in t(w)$. Since $U$ is normal we have without loss of generality that $w < w'$. Let $n$ be such that $v \in \delta(V_n)$. Then $v \leq w < w'$, so since $w' \not \in t(v)$ we have $w' \not \in V_{n+1}$. Since $w \in t(v)$ we have $w \in V_{n+1}$ and so $w' \in V_{n+2}$, so that $v' \in \delta(V_{n+1})$. Since both $v$ and $v'$ lie below $w'$, we have $v < v'$ and so $v$ and $v'$ are adjacent in $T$.
\end{proof}

\section{Determinacy and $(O2)$ for representable trees of matroids}\label{sec:arbi_sums}

We fix a finite field $k$, a $k$-representation $V$ of a tree $\Tcal = (T, M)$ of finite matroids and a set $\Psi$ of ends of $T$, together with a partition $E = \{e\} \dot \cup P_{Co} \dot \cup P_{De}$ of the ground set of $\Tcal$. 
Let $t_0$ be the node of $T$ such that $e \in E(t_0)$.
For a function $f$ whose domain is a subset of $\bigcup_{t\in V(T)}E(t)$, 
we obtain a function $\bar f\colon \bigcup_{t\in V(T)}E(t)\to k$  from $f$ by assigning to each value in $\bigcup_{t\in V(T)}E(t)$ 
but not in the domain of $f$ the value zero.

\begin{dfn}
The {\em circuit game} $\Gcal = \Gcal(V, \Psi, P_{Co}, P_{De})$ is played between two players, called Sarah and Colin, as follows:

Play alternates between the players, with Sarah making the first move. 
At any point in the game there is a {\em current node} $t_c \in V(T)$, 
a {\em current challenge set $S_c\se E(t_c)$}
and a 
{\em current challenge function $x_c\colon S_c\to k$}.
Initially we set $t_c = t_0$, $S_c=\{e\}$ and $x_c(e)=1$. 
For any $n$ the $(2n-1)$\textsuperscript{st} move is made by Sarah: she must play a 
vector $v_n\in V(t_c)$ such that $\bar v_n\restric_{P_{De}}=0$
and $ \bar v_n\not \perp \bar x_c$.
 Then the $2n$\textsuperscript{th} move is made by Colin: he must play a node $t_n$ adjacent to $t_c$ and further away from $t_0$ than $t_c$ is and
a vector $x_n\in k^{E(t_ct_n)}$ 
such that $\bar v_n\not \perp \bar x_n$.
After he does this, the current node is updated to $t_n$, the current challenge set 
to $S_n=E(t_nt_{n-1})$ and the current challenge function to $x_n$. 
If play continues forever, then Sarah wins if the end $\omega$ of $T$ containing $(t_n | n \in \Nbb)$ is in $\Psi$, and Colin wins if $\omega \in \Psi^{\complement}$.

The {\em cocircuit game} $\Gcal^* = \Gcal^*(V, \Psi, P_{Co}, P_{De})$ is the game like the dual circuit game $\Gcal(V^\perp, \Psi\ct, P_{De}, P_{Co})$, but with the roles of Sarah and Colin reversed. We will also use a different notation for the cocircuit game, putting stars on the notation for the circuit game. Thus for example the current challenge function is denoted $x_c^*$ and Colin's $n$\textsuperscript{th} move is denoted $v_n^*$.
\end{dfn}

\begin{lem}\label{Sarah_win}
Sarah has a winning strategy in $\Gcal$ if and only if there is a 
$\Psi$-vector $v$ of $V$ such that $e\in \underline v \se \{e\} \dot\cup P_{Co}$.
\end{lem}

\begin{proof}
Suppose first that there is such a vector $v$. 
Then Sarah can win in $\Gcal$ by always choosing the vector $v(t_c)$ when it is her turn to play. Indeed, for any edge
$tt'\in E(\Tcal)$, the vectors $v(t)$ and $v(t')$ coincide 
when restricted to $E(tt')$.
Hence if 
$ \bar v_n\not\perp \bar x_n$, then also
$\bar v_{n+1}\not\perp \bar x_n$. 
So choosing  $v(t_c)$ is a legal move and since $v$ is a vector, 
the nodes $t_n$ from any play that is played according to this strategy will converge to some 
end in $\Psi$.

Suppose for the converse that Sarah has a winning strategy $\sigma$ in $\Gcal$. 
For each $n$, let $R_n$ be the set of sequences $(v_i|i\leq n)$ which can arise
as the first $n$ moves made by Sarah in a game played according to $\sigma$.

\begin{sublem}\label{sub_lem_Sarah}
Let $r\in R_n$ and let $t(r)$ be the node of $T$ that is current when $r_n$ is played. Let $t'$ be a node of $T$ that is adjacent to $t(r)$ and further
away from $t_0$ than $t(r)$. Let $P_{t'}(r)$ be the set of those $v\in V(t')$
such that the extension $r.v$ of the sequence $r$ by $r_{n+1}=v$ is in $R_{n+1}$.

Then $r_n\restric_{E(t(r)t')}\in \langle v\restric_{E(t(r)t')}|v\in P_{t'}(r)\rangle$.
\end{sublem}

Before proving \autoref{sub_lem_Sarah}, let us see how to derive \autoref{Sarah_win} from it.
By \autoref{sub_lem_Sarah}, for each $r\in R_n$, $t(r)$, $t'$ and $P_{t'}(r)$ as in that Lemma,
we can choose a representation.
\[
 r_n\restric_{E(t(r)t')} = \sum_{v\in P_{t'}(r)} \lambda_{r.v} v\restric_{E(t(r)t')}
\]

Let $r\restric_{i}$ denote the initial sequence of $r$ of length $i$.
For any $t\in V(T)$ at distance $n-1$ from $t_0$, we set:
\[
 v(t)=\sum_{r\in R_n: \ t(r)=t} r_n \cdot \prod_{i=2}^n \lambda_{r\restric_{i}}
\]
 Since each $r_n$ in this expression is in $V(t)$, the vector $v(t)$
is in $V(t)$. And also $e\in \underline v\se  \{e\}\dot\cup P_{Co} $. Next 
we check $t\mapsto v(t)$ is a $\Psi$-vector of $V$. 
For this, we first check that for any $tt'\in E(T)$ with $t'$ further away from $t_0$ than $t$ we have $v(t)\restric_{E(tt')}= v(t')\restric_{E(tt')}$:

\begin{eqnarray*}
v(t)\restric_{E(tt')}&=& \sum_{r\in R_n: \ t(r)=t} r_n\restric_{E(tt')} \cdot \prod_{i=2}^n \lambda_{r\restric_{i}} \\
 &=& \sum_{r\in R_n: \ t(r)=t} \left (\sum_{v\in P_{t'}(r)} \lambda_{r.v} v\restric_{E(tt')}\right ) \cdot \prod_{i=2}^n \lambda_{r\restric_{i}} \\
 &=& \sum_{r\in R_{n+1}: \ t(r)=t'} r_{n+1}\restric_{E(tt')} \cdot \prod_{i=2}^{n+1} \lambda_{r\restric_{i}} \\
 &=& v(t')\restric_{E(tt')}
\end{eqnarray*}

Next, suppose for a contradiction that there is a sequence $t_n$ with the support of $v(t_n)$ nonempty such that its limit is not in $\Psi$. 
Without loss of generality, we may assume that $t_n$ has distance at least $n$ from $t_0$. Hence for each $n\in\Nbb$ there is some $j\geq n$ and some $r\in R_j$ such that $r_j\neq 0$ and $t(r)=t_n$. Since $ 0\perp x$ for every $x$, no $r\restric_i$ can be $0$ for any $i\leq j$ since the play would then be finished after the $i$\supth \ move, which is not true. So without loss of generality, we may assume that $t_n$ has distance precisely $n$ from $t_0$.

Now we apply the Infinity Lemma where we take the $V_n$ from that Lemma to be the sets $\{r\in R_n| t(r)=t_n, r_n\neq0\}$. And we join $r\in R_{n+1}$
to $r'\in R_n$ if and only if $r\restric_{n}=r'$. Note that each $V_n$ is finite since $k$ is finite.
Hence we find a sequence of $a^n\in R_n$ such that $a^{n+1}\restric_{n}=a^n$.
This gives rise to an infinite play according to $\sigma$ whose end is not in $\Psi$,
contradicting the fact that $\sigma$ is a winning strategy.
Thus $t\mapsto v(t)$ is a $\Psi$-vector of $V$.

Having shown how \autoref{Sarah_win} can be deduced from \autoref{sub_lem_Sarah},
it remains to prove \autoref{sub_lem_Sarah}.
For this, we fix a particular finite play of length $2n+1$ according to $\sigma$
and giving rise to $r$, and consider the situation just after this play.
For any $w\in k^{E(t(r)t')}$ with $ \bar w\not\perp\bar r_n$ Sarah 
has a response prescribed by $\sigma$, that is, there is some $v\in P = P_{t'}(r)$ such that
$\bar w\not\perp\bar v$.
In other words, any $w\in k^{E(t(r)t')}$ that is not orthogonal to $x_n$
is also not orthogonal to some $v\in P$.
Put yet another way, any $z\in k^{E(t(r)t')}$ that 
is orthogonal to every $v\in P$ is orthogonal to $x_n$.
By \autoref{span}, $r_n\restric_{E(t(r)t')}\in \langle v\restric_{E(t(r)t')}|v\in P\rangle$.
This completes the proof of \autoref{sub_lem_Sarah}, and so of \autoref{Sarah_win}.
\end{proof}

\begin{cor}
Colin has a winning strategy in $\Gcal^*$ if and only if there is a 
$\Psi\ct$-vector $v^*$ of $V^\perp$ such that $e\in \underline v^* \se \{e\} \dot\cup P_{De}$.
\qed
\end{cor}

In order to relate $(O2)$ to determinacy of $\Gcal$, we need to show that $\Gcal$ and $\Gcal^*$ are closely related games.

\begin{lem}\label{arbi_sumtransfer}
Colin has a winning strategy in $\Gcal$ if and only if he has one in $\Gcal^*$.
\end{lem}

\begin{proof}
For the `if' part, suppose that he has a winning strategy $\sigma^*$ in $\Gcal^*$. Then he can win in $\Gcal$ by playing as follows:

He should imagine an auxilliary play in the game $\Gcal^*$, in which he plays according to $\sigma^*$, and for which he should ensure that at any point the current node 
and current challenge set agree with those in $\Gcal$, and additionally ensure that 
$x_n=v_n^*\restric_{S_n}$ and $x_n^* = v_{n+1} \restric_{S_n}$. We shall assume, without loss of generality, that $v_1(e) = 1$ (otherwise we can just multiply $v_1$ by some constant to make this true).

Suppose Sarah makes some move $v_n$. Then $x_c^* = v_n \restric_{S_{n-1}}$: if $n = 1$ then this is true by our assumption, and otherwise it is true by the condition that $x_n^* = v_{n+1} \restric_{S_n}$.  
Let $v_n^*$ be the move in $\Gcal^*$ that is prescribed by $\sigma^*$. Then $\sum_{f\in S_{n-1}}v_n(f)v_n^*(f) = \sum_{f\in S_{n-1}}x_c^*(f)v_n^*(f)\neq 0$
but $ v_n\perp v_n^*$. Since the support of the map $f\mapsto v_n(f)v_n^*(f)$
consists of dummy edges only, there is some $t_n\in V(T)$ that is adjacent to $t_{n-1}$ and has distance $n$ from $t_0$, such that $\sum_{f\in E(t_{n-1}t_n)}v_n(f)v_n^*(f)\neq 0$.
Then Colin plays $t_n$, $S_n=E(t_{n-1}t_n)$ and $x_n=v_n^*\restric_{S_{n}}$. 
And he plays $v_n^*$ in the imagined cocircuit-game, and imagines that Sarah plays $x_n^* = v_{n+1} \restric_{S_n}$ there. Note that this is a legal move since $\sum_{f \in S_n} v_n^*(f) x_n^*(f) = \sum_{f \in S_n} x_n(f) v_{n+1}(f) \neq 0$. 
If the play of the circuit game continues forever, then the end $\omega$ containing $(t_n | n \in \Nbb)$ is in $\Psi^{\complement}$ since $\sigma^*$ is winning.

For the `only if' part, suppose that he has a winning strategy $\sigma$ in $\Gcal$. 
Then he can win in $\Gcal^*$ by playing as follows:

He should imagine an auxilliary play in the game $\Gcal$, in which he plays according to $\sigma$, and for which he should ensure that at any point the current node 
and current challenge set agree with those in $\Gcal^*$.

When it is his turn to move, either it is his first move, in which case we let $x_0^*$ be the function with support $\{e\}$ that sends $e$ to $1$ or Sarah has just played $x_{n-1}^*$ in $\Gcal^*$.
Then he imagines the corresponding game of $\Gcal$ where he has just played $x_{n-1}$,
or else it is his first move, in which case we set $x_0=x_0^*$.

Let $O$ be the set of Sarah's legal moves in $\Gcal$. For $v\in O$, let 
$t(v)$ and
$x(v)$ be the node and challenge function prescribed by $\sigma$.
Let $T_n=\{t(v)|v\in O\}$.
And for each $t\in T_n$, let $P(t)=\{x(v)|v\in O: t(v)=t\}$.

\begin{sublem}\label{Colin_calc}
There is some $v^*\in V(t_{n-1})^\perp$ and coefficients $\lambda_{t,x}\in k$ and a vector $w \in k^{E(t_{n-1}) \cap P_{De}}$ 
such that
$$\bar x_{n-1}=\overline {v^*}+\overline{w} +\sum_{t\in T_n} \sum_{x\in P(t)} \lambda_{t,x} \bar x.$$
\end{sublem}

Before proving \autoref{Colin_calc}, let us complete the description of his strategy.
In $\Gcal^*$, he plays $v_n^*=v^*$ - by the equation above the support of this vector cannot meet the set $P_{co}$.
Let  $t_n$ and $x_n^*$ be the node and challenge set that Sarah plays in her next move in $\Gcal^*$.
Then  by the choice of $v_n^*$, the node $t_n$ is in $T_n$, and $\bar x_n^*\not \perp \bar v_n^*$. 
Since $v_n^*$ restricted to $E(t_{n-1}t_n)$ is equal to 
$\sum_{x\in P(t_n)} \lambda_{t,x} x$, there is some $x_n\in P(t_n)$
with $x_n\not \perp x_n^*$. Then he imagines that she plays some $v\in O$
with $x(v)=x_n$, and that he then plays $t_n$ and $x_n$.
This completes the description of his strategy.
If play continues forever, then the end $\omega$ containing $(t_n^* | n \in \Nbb)$ is in 
$\Psi^{\complement}$ since $\sigma$ is winning.

Hence it remains to prove \autoref{Colin_calc}.
For this, by \autoref{span}, it remains to show that
 $(V^\perp\cup k^{E(t_{n-1}) \cap P_{De}} \cup \bigcup_{t\in T_n} \bigcup_{x\in P(t)} \bar x)^\perp \se\bar \{x_{n-1}\}^\perp$.
In other words, any $y$ that is not orthogonal to $x_{n-1}$ is not orthogonal
to some $v^*\in V^\perp$ or to some $\bar x$ or has support meeting $P_{De}$.
This follows from the fact that for every $v\in V$ with $v\not \perp x_{n-1}$ and $\underline v \cap P_{De} = \varnothing$,
there is some $x$ such that $v\not\perp \bar x$.
This completes the proof of \autoref{Colin_calc}, and so 
also the proof of \autoref{arbi_sumtransfer}.
\end{proof}

\begin{cor}
$(O2)$ holds for the partition $E = \{e\} \dot\cup P_{Co} \dot\cup P_{De}$ of the groundset of $\Tcal$ if and only if $\Gcal(V, \Psi, P_{Co}, P_{De})$ is determined.\qed
\end{cor}

\begin{cor}
The Axiom of Determinacy is equivalent to the statment that every tree of finite matroids representable over a finite field induces a matroid. \qed
\end{cor}

\begin{cor}\label{Borel_arbi_sums}
For any tree of finite matroids $\Tcal = (T, M)$ represented over a finite field and any Borel set $\Psi$ of ends of $T$, $(\Tcal, \Psi)$ induces a matroid.
\end{cor}
\begin{proof}
Just like the proof of \autoref{Borel2sums}.
\end{proof}

\begin{thm}\label{Borel}
Let $G$ be a locally finite graph, and $\Psi$ a Borel set of ends of $G$. Then $(G, \Psi)$ induces a matroid.
\end{thm}
\begin{proof}
Just like the proof of \autoref{Borelwidth2}.
\end{proof}

\section{From the locally finite case to the countable case}\label{loc_to_count}

\subsection{From the locally finite case to the case that the graph has a locally finite normal spanning tree}

We start with the following construction, which may also be useful in other cases.
Let $G$ be a graph having a normal spanning tree $T$.
Then the \emph{Undomination-graph} $U=U(G,T)$ of $G$ is the following.
Its vertex set is $V(U)=V(G)\times V(T)$.
The pair $(v,t)(v',t')$ is an edge if and only if either 
$v=v'$ and $t$ and $t'$ are adjacent in $T$
or $v$ and $v'$ are adjacent in $G$ and $v=t'$ and $v'=t$.
We call the edges of the first type \emph{$T$-edges} and the ones of the second type \emph{$G$-edges}. We will sometimes implicitly identify the $G$-edge $(v,v')(v',v)$ with the corresponding edge $vv'$ of $G$.

The following properties of $U$ are immediate.
Any vertex of $U$ is incident with at most one $G$-edge. 
$U$ has $G$ as a minor, where
the branching set of the vertex $v$ has the form $\{v\}\times V(T)$. In other words, we obtain $G$ as a minor of $U$ by contracting all $T$-edges.

\begin{dfn}\label{G-path_explicit}
 Let $P_G=p_1(p_1,p_2)p_2\ldots (p_{n-1},p_n)p_n$ be a walk in $G$.
Let $t,t'\in V(T)$. Then $u_{t,t'}(P_G)$ denotes the 
following walk in $U$.
\[
 u_{t,t'}(P_G)= [\{p_1\}\times (tTp_2)] \circ [(p_1,p_{2})(p_{2},p_1)]\circ [\{p_2\}\times (p_1Tp_3)]\circ
\]
\[
[(p_2,p_{3})(p_{3},p_2)]\circ [\{p_3\}\times (p_2Tp_4)]\circ
\ldots \circ [\{p_n\}\times (p_{n-1}Tt')]
\]

\end{dfn}

\begin{dfn}\label{U-path_explicit}
 Let $P_U$ be a walk in $U$ from $(p_1,t)$ to $(p_n,t')$.
Then the set of its $G$-edges forms a walk in $G$ from $p_1$ to $p_n$.
We denote this walk by $g(P_U)$.
\end{dfn}

\begin{lem}\label{inverse}
The operations $u$ and $g$ are inverse to each other for walks that traverse no edge more than once.
\end{lem}

\begin{proof}
It is immediate from the definitions that $g(u_{t,t'}(P))=P$. 

For the other direction, let $P$ be a walk in $U$ from $(p_1,t)$ to $(p_n,t')$.
We are to show that $u_{t,t'}(g(P))=P$.
This follows from the fact the branching set of every $v\in V(G)$ is a tree.
\end{proof}

Note that if $P$ is a path in $G$, then $u_{t,t'}(P)$ is a path whereas
if $P$ is a path in $U$, the walk $g(P)$ need not be a path.

\begin{cor}\label{G-ray_explicit}
 Let $R_G=p_1(p_1,p_2)p_2\ldots $ be a ray in $G$.
Then for any $t\in T$, there is a unique ray $u_t(R_G)$ starting at $(p_1,t)$ in $U$
included in the $T$-edges together with $\{(p_1,p_2)(p_2,p_1),\ldots\}$. 

More precisely:
\[
 u_t(R_G)= [\{p_1\}\times (tTp_2)] \circ [(p_1,p_{2})(p_{2},p_1)]\circ [\{p_2\}\times (p_1Tp_3)]\circ \ldots\]
\end{cor}

\begin{rem}\label{comb_lift}
A result similar to \autoref{G-ray_explicit} also holds for combs
since we have it for paths and rays. A little bit of care is needed when choosing the 
starting points $t$ of the paths $u_{t,t'}(P)$ to ensure that these paths only meet the spine of the comb in their initial vertices.
\end{rem}

The following lemma allows us to turn finite separators in $G$ into finite separators in $U$.

\begin{lem}\label{U_sep_property}
 Let $X$ be a finite set of vertices of $G$, and let $w=(v,t)$ and $w'=(v',t')$
be vertices of $U$ such that $v$ and $v'$ are in different components of $G\sm X$.

Then $X\times X$ separates $w$ from $w'$ in $U$.
\end{lem}

\begin{proof}
Let $P_U$ be some $w$-$w'$-path in $U$. 
Let $g(P_U)=p_1(p_1,p_2)p_2\ldots (p_{n-1},p_n)p_n$ with $p_1=v$
and $p_n=v'$.

Let $C_1$ be the component of $G\sm X$
containing $p_1$. Let $i\in \{1,\ldots,n\}$ be maximal such that $p_i\in C_1$.
Such an $i$ exists as $p_1\in C_1$.
Note that $p_i\neq p_n$.
Then $p_{i+1}$ is in $X$.

Since $p_{i+1}\neq p_n,p_1$, the path $P_U$ has $\{p_{i+1}\}\times (p_iTp_{i+2})$ as a subpath.
Since $p_i\in C_1$ but $p_{i+2}\notin C_1$, the path $p_iTp_{i+2}$ has to meet $X$ in some point $x$.
Then $P_U$ meets $X\times X$ in $(p_{i+1},x)$, completing the proof.
\end{proof}

The following lemma is the reason why we call $U$ the Undomination-graph of $G$.

\begin{prop}\label{U_no_dom}
 In $U(G,T)$, no vertex dominates a ray.
\end{prop}

\begin{proof}
Suppose for a contradiction that $U$ has a vertex $(v,t)$ dominating a ray $R$.
Then there is an infinite collection $(P_n|n\in \Nbb)$ of $(v,t)$-$R$-paths in $U$
that meet only in $(v,t)$. Since all edges except for at most one edge incident with  
$(v,t)$ are $T$-edges, we may assume that the second vertex on each $P_n$ has the form
$(v_n,t)$ where $v_n$ is an neighbour of $v$ in $T$. Since $v$ has at most one lower neighbour in $T$, we may even assume that all the $v_n$ are upper neighbours of $v$ in $T$.

Let $\lceil v\rceil$ be the set of those vertices that are less than or equal to $v$ in the 
tree order of $T$.
As $T$ is normal, all the $v_n$ are in different components of $G\sm \lceil v\rceil$.
By \autoref{U_sep_property}, all the $(v_n,t)$ are in different components of 
$U\sm (\lceil v\rceil\times \lceil v\rceil)$. 

Since $\lceil v\rceil\times \lceil v\rceil$ is finite,
we can find a tail $R'$ of $R$ that avoids $\lceil v\rceil\times \lceil v\rceil$.
Then for any two paths $P_i$ and $P_j$ that avoid $\lceil v\rceil\times \lceil v\rceil$
and meet $R'$, the set $R'\cup P_i \cup P_j$ is connected in $U\sm (\lceil v\rceil\times \lceil v\rceil)$.
Hence the vertices $(v_i,t)$ and $(v_j,t)$ are in the same connected component of $U\sm (\lceil v\rceil\times \lceil v\rceil)$.

Since $R\sm R'$ and $\lceil v\rceil\times \lceil v\rceil$ are both finite, there exist such paths $P_i$ and $P_j$, which yields the desired contradiction.
\end{proof}

Next we shall investigate how the ends of $U$ relate to the ends of $G$.

\begin{lem}\label{f_nice}
 Let $R_1$ and $R_2$ be rays of $G$.
Then $R_1$ and $R_2$ belong to the same end of $G$ if and only if
$u_t(R_1)$ and $u_{t'}(R_2)$ belong to same end of $U$ for any $t,t'\in V(T)$.
\end{lem}

\begin{proof}
First suppose that $u_t(R_1)$ and $u_{t'}(R_2)$ belong to different ends of $U$.
Then there is a finite set $S=(v_1,t_1),\ldots,(v_n,t_n)$ separating them.
Without loss of generality we may assume that $u_t(R_1)$ and $u_{t'}(R_2)$ do not meet $S$.
Let $P$ be some $R_1$-$R_2$-path, which goes from $(w,s)$ to $(w',s')$.
Then $u_{s,s'}(P)$ meets $S$ in some point, say $(v_i,t_i)$.
Hence $\{v_1,\ldots,v_n\}$ separates $R_1$ from $R_2$, yielding the first implication.

The other implication is an immediate consequence of \autoref{U_sep_property}.
\end{proof}

By \autoref{f_nice}, the map $u$ induces an inclusion $\tilde u$ from the ends of $G$ into the ends of $U$. 
Let $C$ be the set of $T$-edges.
The purpose of this subsection is to prove the following.

\begin{thm}\label{loc_fin_to_fin_sep}
Assume that $(U,{\tilde u(\Psi)})$ induces a matroid $M$.
Then $(G,\Psi)$ induces the matroid $M/C$.
\end{thm}

The Undomination-graph $U(G,T)$ is locally finite whenever $T$ is locally finite.
Thus \autoref{loc_fin_to_fin_sep}, reduces the case where $G$ has a locally finite normal spanning tree to the locally finite one, which is the aim of this subsection.

The proof of \autoref{loc_fin_to_fin_sep} takes the rest of this subsection.

\begin{lem}\label{loc_fin_to_fin_sep_bonds}
Assume that $(U,{\tilde u(\Psi)})$ induces a matroid $M$.
Then the edge set $b$ is an $M/C$-cocircuit $b$ if and only if it is a $\Psi\ct$-bond of $G$.
\end{lem}

\begin{proof}
First suppose that $b$ is an $M/C$-cocircuit.
 The cocircuit $b$ is a $\tilde u(\Psi)\ct$-bond of 
$U$ that does not meet $C$.
Since the graphs $U/C$ and $G$ are equal, it remains to show that $b$ considered as an edge set of $G$ does not have any end of $\Psi$ in its closure.

 Suppose for a contradiction that there is such an end $\omega\in \Psi$ that is in the closure of $b$. Let $R_\omega$ be some ray in $\omega$.

By \autoref{spine}, there is a comb $K$
with spine $R_\omega$ all of whose teeth are endvertices of $b$.
Then in $U$, the set $K\cup C$ contains a comb all of whose teeth are in endvertices of $b$
with spine $u_t(R_\omega)$ for some $t$ by \autoref{comb_lift}.
Hence $\tilde u(\omega)$ is in the closure of $b$, a contradiction.

Next suppose that $b$ is a $\Psi\ct$-bond of $G$.
As above, it is clear that $b$ considered as an edge set of $U$
is a bond. 

Now suppose for a contradiction that there is some end $\omega\in \tilde u(\Psi)$
in the closure of $b$. We pick a ray $R_\omega\in \tilde u^{-1}(\omega)$.
By \autoref{spine} there is a comb in $U$ with spine $u(R_\omega)$ all of whose teeth are endvertices of $b$. Then this comb defines a comb in $G$ with comb $R_\omega$, which is impossible. This completes the proof.
\end{proof}

Next we prove \autoref{loc_fin_to_fin_sep_bonds} for circuits, which is a little more complicated.

We define the map $p:|U|_{\tilde u(\Psi)}\to |G|_\Psi$ as follows.
A vertex $(v,t)$ maps to $v$, a $G$-edge $(v,t)(t,v)$ maps to the edge $vt$,
all interior points of a $T$-edge $(v,t)(v,t')$ map to $v$, and an end $\omega\in \tilde u(\Psi)$ maps to
$\tilde u^{-1}(\omega)$.

\begin{lem}\label{conti}
 $p$ is continuous.
\end{lem}

\begin{proof}
Let $O$ be some open set in $|G|_\Psi$.
Let $x\in p^{-1}(O)$. 
If $x$ is an interior point of a $G$-edge, then $p^{-1}(O)$ clearly includes a neighbourhood around $x$.
If $x$ is an interior point of a $T$-edge, then $p^{-1}(O)$ included
the whole interior of that edge. 

If $x$ is a vertex, then there is some $\epsilon$ with $B_{\epsilon}(p(x)) \subseteq O$: then $B_{\epsilon}(x) \subseteq p^{-1}(O)$.

If $x$ is an end, then some basic open set $ \hat C(S,p(x))$
 is included in $O$.
Let $D=D(S\times S,x)$ be the unique component of $U\sm S\times S$ having $x$ in its closure.
We show that $ \hat D(S\times S,x)$ is a subset of $p^{-1}(O)$.
Clearly all edges and vertices of $ \hat D(S\times S,x)$ are in $p^{-1}(O)$.
So let $\omega\in \hat D(S\times S,x)$ be an end.

Let $(v,t)\in D$.
Let $R$ be a ray in $G$ that is in $\tilde u^{-1}(\omega)$.
Then (for any $t$) $u_t(R)$ is eventually in $D$
as $u_t(R)\in\omega$. 
By  \autoref{U_sep_property}, $R$ is then eventually in the same component as $v$.
So it is in $C(S,p(x))$.
Hence $\omega\in \hat C(S,p(x))$.
This completes the proof of the continuity of $p$.
\end{proof}

Since $\tilde G_\Psi$ has the quotient topology, the quotient map $\pi_G:|G|_\Psi\to\tilde G_\Psi$ is continuous. Similarly, the quotient map 
$\pi_U:|U|_{\tilde u(\Psi)}\to\tilde U_{\tilde u(\Psi)}$ is continuous.
All the maps occurring here are shown in \autoref{fig:maps}.

\begin{lem}\label{tilde_p_conti}
For any two $x,y\in \tilde  U_{\tilde u(\Psi)}$ 
with $\pi_G(p(x))\neq \pi_G(p(y))$, we have $\pi_U(x)\neq \pi_U(y)$.

In particular, there is a unique map $\tilde p: \tilde U_{\tilde u(\Psi)}\to \tilde G_\Psi$
satisfying $\tilde p(\pi_U(x))=\pi_G(p(x))$.
Moreover, $\tilde p$ is continuous.
 \end{lem}

It might be worth noting that since $U$ is locally finite, the map $\pi_U$ is the identity,
which makes the Lemma rather trivial. 
However we will not use this in the proof as we rely on this Lemma later on  
in a slightly different context where $\pi_U$ is not the identity.

\begin{proof}
Since $\pi_G(p(x))\neq \pi_G(p(y))$, there is some
$\Psi$-bounded cut of $G$ with $p(x)$ and $p(y)$ on different sides by \autoref{vertices_equivalent}.
Then there is also a $\Psi\ct$-bond $b$ of $G$ with $p(x)$ and $p(y)$ on different sides.
By \autoref{loc_fin_to_fin_sep_bonds}, the bond $b$ is also a 
$\tilde u(\Psi)\ct$-bond in $U$. And this bond witnesses that
$\pi_U(x)\neq \pi_U(y)$ by the other implication of \autoref{vertices_equivalent}.
This proves the first part of the Lemma.

It remains to show that $\tilde p$ is continuous.
This follows from the universal property of the quotient map $\pi_U$
since the concatenation of $\pi_G$ and $p$ is continuous.
\end{proof}

\begin{cor}\label{loc_fin_to_fin_sep_circuits}
Assume that $(U,{\tilde u(\Psi)})$ induces a matroid $M$.
Then for any $M/C$-circuit $o$ and any edge $e\in o$, the circuit $o$ includes a $\Psi$-circuit of $G$ containing $e$.
\end{cor}

\begin{proof}
Let $o$ be some $M/C$-circuit. Then there is some $M$-circuit $o\se o'\se o\cup C$
by \autoref{rest_cir}. 
Let $o''=\tilde p\circ o'$ as in \autoref{fig:maps}.

Let $e$ be some edge in $o$. 
Then $e$ considered as an edge of $U$ is mapped under $\tilde p$
to the edge $e$ considered as an edge of $G$, which is then in the image of $o''$. 

Then the restriction of $o''$ to those points that do not map to interior points of $e$
is a path between the two endvertices of $e$, that is a continuous function from $[0,1]$
to $\tilde G_\Psi$ mapping $0$ and $1$ to the endvertices of $e$.
By a well-known Lemma of basic topology\cite{Armstrong}, there is an arc (injective path) between the two endvertices of $e$ whose image is included in the image of that path. 
The concatenation of this arc with some continuous function from $[0,1]$ to $e$ defines the desired $\Psi$-circuit.
\end{proof}

\begin{figure}
\begin{center}
 \includegraphics[height=4cm]{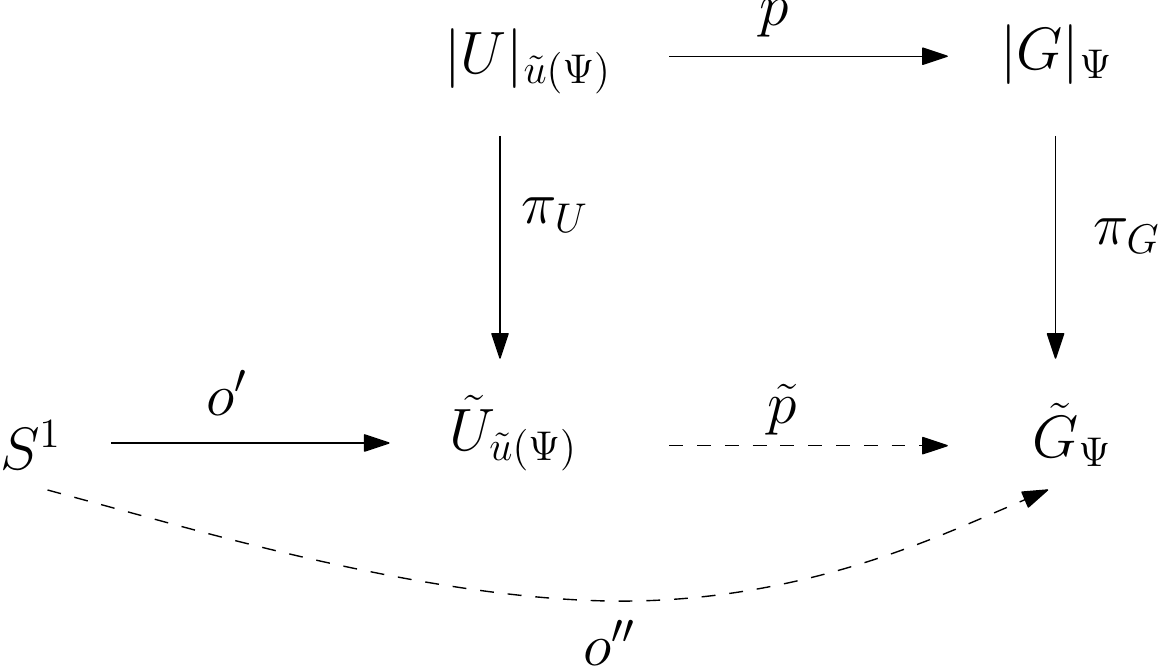}
\end{center}
\caption{The construction of the map $o''$.}\label{fig:maps}
\end{figure}

By \autoref{loc_fin_to_fin_sep_circuits}, \autoref{loc_fin_to_fin_sep_bonds}
and \autoref{never_meet_just_once}, we can apply \autoref{cir_c_scrawl} and deduce
\autoref{loc_fin_to_fin_sep}.

\subsection{From the case that the graph has a locally finite normal spanning tree to the countable case}

The aim of this subsection is to prove the following.

\begin{prop}\label{fin_sep_to_countable}
For every countable graph $G$ together with $\Psi_G\se \Omega(G)$ there is a graph $H$
having a locally finite normal spanning tree together with $\Psi_H\se \Omega(H)$  and $C\se E(H)$ such that
if $(H,{\Psi_H})$ induces a matroid $M$, then $(G,{\Psi_G})$ induces $M/C$.
\end{prop}

First we need the following lemma.

\begin{lem}\label{fin_sep_to_countable_simply}
Let $G$ be a countable graph together with a normal spanning tree $T_G$.
Then there is a countable graph $H$ 
together with a locally finite normal spanning tree $T_H$ and $C\se E(T_H)$ such that
$G=H/C$ and $T_G=T_H/C$.
\end{lem}

\begin{proof}[Proof of \autoref{fin_sep_to_countable_simply}.]
 First we construct $T_H$. Let $X$ be the set of those vertices of $T_G$ that have infinitely many upper neighbours. We obtain the tree $T'$ from $T_G$ by adding a ray $R_x$ starting at $x$ for every $x\in X$.

We obtain $T_H$ from $T'$ by replacing each edge of the type $vx$ where $v$ is an upper neighbour in $T'$ of $x$ by the edge $vx'$ for some $x'\in R_x-x$ in such a way that for all $x\in X$ all 
vertices in $R_x-x$ get degree $3$. This is possible by the choice of $X$.
Note that any $x\in X$ has degree at most $2$ in $T_H$, and hence $T_H$ is locally finite.
Let $C$ be the set of all those edges contained in some $R_x$. Then $C\se E(T_H)$ and $T_G=T_H/C$.

Note that $V(G)\se V(T_H)$. We obtain $H$ from $T_H$ by adding all edges $e\in E(G)\sm E(T_G)$.
It is straightforward to check that $G=H/C$ and $T_H$ is normal in $H$.
This completes the proof.
\end{proof}

\begin{proof}[Proof of \autoref{fin_sep_to_countable}.]
First note that every countable graph has a normal spanning tree \cite{DiestelBook10}.
Hence we may pick a normal spanning tree $T_G$ of $G$.

By \autoref{fin_sep_to_countable_simply}, 
there is a countable graph $H$ 
together with a locally finite normal spanning tree $T_H$ and $C\se E(T_H)$ such that
$G=H/C$ and $T_G=T_H/C$.

Every normal ray $R$ of $G$ starting at some vertex $v\in V(G)$ extends to a unique normal ray $h(R)$ starting at the same vertex $v$ and that is included in $R\cup C$.

It is straightforward to check that $R$ and $R'$ belong to the same end of $G$ if and only if $h(R)$ and $h(R')$ belong to the same end of $H$.
This defines an inclusion $\tilde h$ from the ends of $G$ into the ends of $H$.
We let $\Psi_H=\tilde h(\Psi_G)$. 

We define the map $p:|H|_{\Psi_H}\to |G|_{\Psi_G}$ 
to be $\tilde h^{-1}$ on the ends, map $R_x$ to $x$ for every $x\in X$, and to be the identity everywhere else.
As in the proof of \autoref{loc_fin_to_fin_sep}, we show the following.

\begin{lem}\label{conti_count}
 $p$ is continuous.
\end{lem}

\begin{proof}
Let $O$ be some open set in $|G|_{\Psi_G}$.
Let $y\in p^{-1}(O)$.
If $y$ is a vertex or an interior point of an edge, then
$p^{-1}(O)$ includes an open neighbourhood around $y$ as in the proof of \autoref{conti}.
 
If $y$ is an end in $\Psi_H$, 
then $O$ includes a basic open set of the form $ \hat C(S,\tilde h^{-1}(y))$.
We pick $v\in V(G)$ such that in $T_G$ it separates $S$ from $\tilde h^{-1}(y)$.
Note that this is possible since $S$ is finite.

Then $\hat C(S_v,y)\se p^{-1}(O)$ where $S_v$ is the down-closure of $v$ in $T_H$.
This completes the proof of the continuity of $p$.
\end{proof}

Now assume that $(H,{\Psi_H})$ induces a matroid $M$. 
The proofs of Lemmata \ref{loc_fin_to_fin_sep_bonds}, \ref{tilde_p_conti} and \ref{loc_fin_to_fin_sep_circuits} extend immediately to our setting.
Hence we can apply the proof of \autoref{loc_fin_to_fin_sep} from the last section
to conclude that $(G,{\Psi_G})$ induces $M/C$.
\end{proof}


\section{Applications}\label{apps}

We are now in a position to begin applying our main results, to answer some of the basic questions about matroids discussed in the introduction. We begin by showing that there are as many countable tame matroids as there could possibly be: we prove that there are $2^{2^{\aleph_0}}$ non-isomorphic countable tame matroids with no $M(K_4)$-minor and no $U_{2,4}$-minor (\autoref{non_iso_cor} from the Introduction).

\begin{proof}
First we outline the construction of the $2^{2^{\aleph_0}}$ non-isomorphic matroids.
 Let $T$ be a tree with precisely one vertex of each finite degree $\geq 2$.
We will use the graph $G = T \times K_2$: that is, $G$ is built from two disjoint copies of $T$ by adding an edge between each vertex and its clone. We call the two copies of the vertex of degree $n$ $v_n$ and $v_n'$.
So for any $\Psi\se \Omega(G)$ the pair $(G,\Psi)$ induces a matroid $M(\Psi)$ 
by \autoref{T_2_lem}.

It suffices to show for any isomorphism $f:M(\Psi)\to M(\Psi')$ that $\Psi = \Psi'$.

The edge $v_nv_n'$ is in precisely $n$ circuits of length $4$. 
Since all edges not of the type $v_nv_n'$ are in precisely one circuit of length $4$,
the map $f$ maps $v_nv_n'$ to itself.

Let $e$ be some edge not of the type $v_nv_n'$. Then it is contained in a unique circuit of length $4$ which contains its clone and two other edges, say $v_iv_i'$ and $v_jv_j'$.
The edge $e$ cannot be distinguished from its clone and $f$ may map it to itself or to its clone but it cannot map it to some other edge because $f(e)$ must lie in a common $4$-circuit
with $v_iv_i'$ and $v_jv_j'$.

For every end $\omega$ of $G$, $\omega \in \Psi$ if and only if the unique double ray $D$ containing $v_2v_2'$ and with both ends in $\omega$ is a circuit of $M(\Psi)$. But by the above argument, each such $D$ is fixed by $f$.
Hence $\Psi = \Psi'$. 

Having shown that there are $2^{2^{\aleph_0}}$ non-isomorphic tame matroids,
it remains to show that none of them has $M(K_4)$ or $U_{2,4}$ as a minor.
Combining the fact that in these matroids every circuit-cocircuit intersection is even
by \autoref{tree_thin_sum}
with a result of \cite{BC:rep_matroids}, yields that they do not have a $U_{2,4}$-minor.

If $M(\Psi)$ had an $M(K_4)$-minor, we would be able to find a 2-separation of $G$ with at least two of the six edges of that minor on each side. But this would induce a 2-separation of $M(\Psi)$, which in turn would induce a separation of $M(K_4)$ with at least 2 edges on each side. Since there is no such 2-separation, there can be no $M(K_4)$ minor.
\end{proof}

A direct consequence of \autoref{non_iso_cor} is that there is no universal matroid for the class of countable planar matroids (\autoref{uni_cor} from the Introduction) since every countable matroid has at most $2^\omega$ many non-isomorphic minors
but the class of countable planar matroids has $2^{2^{\aleph_0}}$ many non-isomorphic members.

Finally, we prove that the countable binary matroids of branch-width at most $2$ are not well-quasi-ordered (\autoref{wqo} from the Introduction).

Throughout the rest of this section $2^\Nbb$ is endowed with the product topology. The next Lemma finds complicated subsets of $2^\Nbb$.

\begin{lem}\label{wqo_lem}
 There is a sequence of subsets $\Psi_n \se 2^\Nbb$ with the following properties.
\begin{enumerate}
 \item Each $\Psi_n$ has cardinality $2^{\aleph_0}$.  \label{cond1}
\item There do no not exist $i<j\in \Nbb$ and an injective continuous map $f:2^\Nbb\to 2^\Nbb$ such that $f(\Psi_i)\se \Psi_j$. \label{cond2}
\end{enumerate}
\end{lem}

Before proving this lemma, let us see how we can deduce \autoref{wqo} from it.

\begin{proof}[Proof that \autoref{wqo_lem} implies \autoref{wqo}.]
As in \autoref{T_2_lem}, we consider the graph $G=T_2\times K_2$.
Note that $\Omega(G)$ and $2^\Nbb$ are homeomorphic.
Let $M_n$ be the $\Psi$-matroid $M(G_n,\Psi_n)$ with $G_n=G$, which is a matroid as shown in that Lemma.
It is easy to check that $G$ has branch-width $2$, so $M_n$ has branch-width $2$ as well.

Suppose for a contradiction that there are $i<j$ such that $M_i \cong M_j/C\backslash D$.
By \autoref{wqo_lem}, it remains to find an injective continuous map $f\colon \Omega(G_i)\to \Omega(G_j)$ such that $f(\Psi_i)\se \Psi_j$. 

For $\omega\in \Omega(G_i)$, we pick a double ray $D(\omega)$ having only $\omega$ in its closure. Then the edge set of $D(\omega)$ considered as an edge set of $G_j$ has only a single end in its closure. Indeed, if there were two ends in its closure, then there is a  $2$-separation of $M_j$ having infinitely many edges from $D(\omega)$ on both sides. This then would give rise to a $2$-separation of $M_i$ with infinitely many edges from $D(\omega)$ on both sides, which is impossible.

This motivates the following definition: we define $f(\omega)$ to be the unique end of $G_j$ in the closure of $D(\omega)$. Note that this does not depend on the choice of $D(\omega)$
since any two such choices differ by finitely many edges only.

To see that $f(\Psi_i)\se \Psi_j$, note that for every $\omega\in\Psi_i$ the set $D(\omega)$ 
extends to a circuit of $M_j$ using additionally only edges from $C$. This circuit has only ends from $\Psi_j$ in the closure. Hence the unique end in the closure of $D(\omega)$ must be in $\Psi_j$.

To see that $f$ is continuous, let $\omega\in \Omega(G_i)$ and let $\hat C_\epsilon(S,f(\omega))$ be a basic open neighbourhood of $f(\omega)$. 
Then $S$ defines a separation of $M_j$ of finite order with the edges of $C(S,\omega)$ on one side.
Then the set $F$ of all these edges without $C\cup D$ forms the side of a separation of finite order in $M_i$, which gives rise to a vertex separator $S'$ in $G_i$ (Formally, $S'$ consists of those vertices that are incident with one edge in $F$ and one outside). Then 
$\hat C_\epsilon(S',\omega)\se f^{-1}(\hat C_\epsilon(S,f(\omega)))$.
Hence $f$ is continuous.

It remains to show that $f$ is injective.
So suppose for a contradiction that there are $\omega_1\neq\omega_2$ in $\Omega(G_i)$
that are mapped to the same end $\tau$ in $\Omega(G_j)$.
We may assume that we picked $D(\omega_1)$ and $D(\omega_2)$ such that they are 
vertex-disjoint.

We shall construct a $2$-separation $(A,B)$ of $M_j$ such that 
$A$ and $B$ both include an edge from each of $D(\omega_1)$ or $D(\omega_2)$.
For this, we pick some $e_1\in D(\omega_1)$ and some $e_2\in D(\omega_2)$.
Then in $G_j$, there are two vertices $v$ and $w$ such that 
the components of $G-v-w$ containing $e_1$ or $e_2$ do not have $\tau$ in their closure.
Let $B$ consist of those edges of $G$ that are only 
incident with $v$, $w$ or vertices of the component $G/\{v,w\}$ that has $\tau$ in its closure.
Let $A=E(M_j)\sm B$.

Since $A\sm(C\cup D)$ and $B\sm(C\cup D)$ both have at least $2$ elements,
 $(A\sm(C\cup D), B\sm(C\cup D) )$ is a $2$-separation of $M_j/C\backslash D$.
Since $M_i \cong M_j/C\backslash D$, this gives rise to a $2$-separation
of $M_i$ having on each side at least one edge from each of $D(\omega_1)$ and $D(\omega_2)$.

This gives rise to a $2$-separation $(A',B')$ in $G_i$, and it induces a separation on the closure of $D(\omega_1)$ in $M_i$. 
Since this closure is $2$-connected, this separation has order $2$
and thus $D(\omega_1)$ includes the separator of $(A',B')$.
Similarly,  $D(\omega_2)$ includes this separator,
contradicting the fact that $D(\omega_1)$ and $D(\omega_2)$ are vertex-disjoint.
This completes the proof.

\end{proof}

\begin{proof}[Proof of \autoref{wqo_lem}.]

We build the sets $\Psi_n$ recursively. So let us suppose that $\Psi_1,\ldots,\Psi_n$
are already constructed such that they satisfy (\ref{cond1}) and (\ref{cond2})
for all $j\leq n$.

Let $K$ be the set of pairs $(i,f)$ where $i\leq n$ and $f:2^\Nbb\to 2^\Nbb$
is a continuous injective function.

Since $2^\Nbb$ has a countable basis as a topological space, the set $K$ has size $2^{\aleph_0}$. Let $\kappa$ be the least ordinal
of size $2^{\aleph_0}$.
We can well-order $K$ as $((i_\alpha,f_\alpha)|\alpha<\kappa)$.

For every $\alpha<\kappa$ we pick two elements $s_\alpha,t_\alpha\in 2^\Nbb$
such that all the $s_\alpha$ and $t_\alpha$ are distinct and
$t_\alpha\in  f_\alpha(\Psi_{i_\alpha})$.
This is possible as $|2^\Nbb|=2^{\aleph_0}$.

We let $\Psi_{n+1}=\{s_\alpha|\alpha<\kappa\}$.
Then $|\Psi_{n+1}|=2^{\aleph_0}$ since all the $s_\alpha$ are disjoint.
Let $f$ be some continuous function $f:2^\Nbb\to 2^\Nbb$ and $i<n+1$.
Then there is some $\alpha<\kappa$ such that $f_\alpha=f$ and $i_\alpha=i$.
We ensured at step $\alpha$ that $t_\alpha\in f_\alpha(\Psi_{i_\alpha})$ and hence
$f(\Psi_i)\not\se \Psi_{n+1}$,
yielding (\ref{cond2}) for all $j\leq n+1$.
This completes the proof.
\end{proof}

\section{Acknowledgements}

We are grateful to Julian Pott for some useful discussions of an earlier version of these ideas.

\bibliographystyle{plain}
\bibliography{literatur}

\end{document}